\documentclass[a4paper,12pt]{article}

\usepackage[ngerman, american]{babel} 
\usepackage[utf8]{inputenc}
\usepackage[a4paper, left=2.3cm, right=2.3cm, top=2.5cm, bottom=2.5cm]{geometry}
\usepackage{graphicx}
\usepackage{float}
\usepackage{color,psfrag}
\usepackage{amssymb}
\usepackage{amsmath} 
\usepackage{amsthm}
\usepackage{dsfont}
\usepackage{mathrsfs}
\usepackage{color}
\usepackage{setspace}
\usepackage{verbatim}
\usepackage{tikz}
\usetikzlibrary{decorations.pathmorphing,decorations.pathreplacing,decorations.shapes}
\tikzset{different1/.style={decorate, decoration=border}}
\usepackage{capt-of}
\usepackage{caption}
\usepackage{hyperref}
\renewcommand{\epsilon}{{\varepsilon}}        
          
\renewcommand{\phi}{{\varphi}}

\newcommand{\C}{\mathds{C}}
\newcommand{\R}{\mathds{R}}
\newcommand{\N}{\mathds{N}}

\newcommand{\rIm}{\mathop{\mathrm{Im}}}
\newcommand{\rRe}{\mathop{\mathrm{Re}}}

\newcommand{\tr}{\mathrm{tr}}

\newcommand{\E}{\mathds{E}}
\renewcommand{\P}{\mathds{P}}
\newcommand{\kap}[1]{\kappa^{(#1)}}
\newcommand{\re}{\mathrm{e}}

\newtheorem{theorem}{Theorem}    
\newtheorem{lemma}[theorem]{Lemma}             

\newtheorem{definition}[theorem]{Definition}

\newtheoremstyle{mystyle}
  {}
  {}
  {\itshape}
  {}
  {\bfseries}
  {}
  { }
  {\thmname{#1}\thmnumber{ #2}\thmnote{ (#3)}}

\theoremstyle{mystyle}
\newtheorem*{assumption}{Assumption}
\newtheoremstyle{myrem}
  {}
  {}
  {}
  {}
  {\bfseries}
  {.}
  { }
  {\thmname{#1}\thmnumber{ #2}\thmnote{ (#3)}}
\theoremstyle{myrem}
\newtheorem{example}[theorem]{Example}
\newtheorem*{crule}{Counting Rule}
\newtheorem*{remark}{Remark}

\newcounter{assumptions}
\newcommand{\labelA}[1]{\refstepcounter{assumptions}\label{#1}}

\newcounter{rules}

\newcommand{\labelR}[1]{\refstepcounter{rules}\label{#1}}

\providecommand{\keywords}[1]{\textbf{Keywords: }{#1}.}
\providecommand{\AMSclass}[1]{\textbf{AMS Subject Classification (2020): }{#1}.}
\begin{document}
\setcounter{page}{0} \thispagestyle{empty}
\title{\vspace{-2cm}On the operator norm of a Hermitian random matrix with correlated entries}
\author{Jana Reker, IST Austria}

\maketitle
\date
\begin{abstract}
We consider a correlated $N\times N$ Hermitian random matrix with a polynomially decaying metric correlation structure. By calculating the trace of the moments of the matrix and using the summable decay of the cumulants, we show that its operator norm is stochastically dominated by one.
\end{abstract}

\AMSclass{60B20}\\

\setcounter{page}{1}
\vspace{-2mm}
\section{Introduction}
Let $H$ be a Hermitian $N\times N$ random matrix such that $H=\smash{\frac{1}{\sqrt{N}}}W$, where $W\in\C^{N\times N}$ has matrix elements of order one. For Wigner matrices, i.e., when the entries of $W$ are identically distributed and independent (up to the Hermitian symmetry) with some mild moment condition, it is well-known that $\| H\|$ is bounded uniformly in $N$ with very high probability. In fact, it even converges to 2 under the normalization $\E |W_{ij}|^2=1$ (see~\cite{BaiYin1988} and~\cite[Thm~2.1.22]{AndersonGuionnetZeitouni2009}, as well as~\cite{FuerediKomlos1981},~\cite{Juhasz1981},~\cite{Vu2007} for more quantitative bounds under stronger moment conditions). Similar statements hold true for Wigner-type ensembles which allow different distributions of the matrix entries and a more general variance profile, but retain independence up to the Hermitian symmetry. Here, the operator norm converges to the maximum of the support of the asymptotic density of states (see~\cite{ErdosMuehlbacher2019},~\cite{Ottolini2017}), which, although possibly different from~2, is deterministic as well. In contrast, if the entries of~$W$ are very strongly correlated, the norm of $H$ may be as large as~$\sqrt{N}$. Considering dependent random variables furher adds the technical difficulty of tracking the correlations between the matrix entries explicitly throughout the analysis. The spectral properties of correlated random matrices have been studied through global laws (see, e.g.,~\cite{AndersonZeitouni2008},~\cite{BannaMerlevedePeligrad2015},~\cite{BoutetdeMonvelKhorunzhyVasilchuk1996},~\cite{HachemLoubatonNajim2005},~\cite{RashidiFarOrabyBrycSpeicher2008},~\cite{SchenkerSchulzBaldes2005} for the discussion of different models and correlation structures), local laws and universality results~(see~\cite{AdhikariChe2019},~\cite{AEKS2020},\cite{ErdosKruegerSchroeder2017}).

\medskip
In this paper, we assume the entries of~$W$ to be correlated following a \emph{polynomially decaying metric correlation structure} as, e.g., considered in~\cite{ErdosKruegerSchroeder2017}, but with a weaker, merely summable correlation decay. This dependence structure is characterized by the 2-cumulants of the matrix elements~$W_{ij}$ decaying at least as an inverse $2+\epsilon$ power of the distance with respect to a natural metric on the index pairs $(i,j)$. The higher cumulants follow a similar pattern~(see Assumption~\eqref{A3} below). Under these mild decay conditions, we show that $\| H\|$ is essentially bounded with very high probability.

\medskip
This result was already stated in~\cite{ErdosKruegerSchroeder2017}, indicating that an extension of Wigner’s moment method applies. In the current paper, we carry out this task which turns out to be rather involved. In~\cite{ErdosKruegerSchroeder2017}, this bound was used as an apriori control on $\| H\|$ for the resolvent method, leading to optimal local laws for~$H$. We remark that it is possible the modify the proof in~\cite{ErdosKruegerSchroeder2017} to obtain the bound on $\| H \|$ directly, i.e., without relying on the current paper. However, an independent proof via the moment method has several advantages. First, it is conceptually much simpler and less technical than the resolvent approach in  \cite{ErdosKruegerSchroeder2017}. Further, it only requires the summability of the 2-cumulants (see exponent $s>2$ in~\eqref{A3-2-cumu} below), while \cite{ErdosKruegerSchroeder2017} assumed a faster decay ($s>12$, see~\cite[Eq.~(3a)]{ErdosKruegerSchroeder2017}). Lastly, the current method can be generalized to even weaker correlation decays, resulting in correlated random matrix ensembles whose norm grows with~$N$ but slower than the trivial~$\sqrt{N}$ bound.

\medskip
We start by giving some general notation and the precise assumptions on the matrix $W$ in Section~\ref{sect-notation} below. The bound on the operator norm of $H$ is then formulated in Theorem~\ref{main} and its proof is given in Sections~\ref{sect-2-cumu-estimates} and~\ref{sect-hi-cumu-estimates}. For simplicity, the argument is carried out only for symmetric $H\in\R^{N\times N}$. The Hermitian case follows analogously and is hence omitted.

\medskip
\textbf{Acknowledgment:}
I am very grateful to László Erdős for suggesting the topic and supervising my work on this project. Many thanks also to Lorenz Hübel for pointing out two mistyped estimates in the published version of this article. Partially supported by ERC Advanced Grant "RMTBeyond" No.~101020331.

\vspace{-2mm}
\subsection{Notation and Assumptions on the Model}\label{sect-notation}
Throughout the paper, boldface indicates vectors $\mathbf{x}\in\C^N$ and their Euclidean norm is denoted by $\|\mathbf{x}\|_2$. Further, the operator norm of a matrix~$A\in\C^{N\times N}$ is denoted by~$\|A\|$. In the estimates, $C$ (without subscript) denotes a generic constant the value of which may change from line to line. We note the following assumptions on the matrix $W$.

\begin{assumption}[A1]\labelA{A1}
$\E W_{i,j}=0$ for all $i,j=1,\dots,N$.
\end{assumption}
\begin{assumption}[A2]\labelA{A2}
For all $q\in\N$ there exists a constant $\mu_q$ such that $\E|W_{i,j}|^q\leq\mu_q$ for all~$i,j=1,\dots,N$.
\end{assumption}

The assumption on the correlation decay is given in terms of the multivariate cumulants~$\kap{k}$ of the matrix elements.
\begin{definition}[Cumulants]
Let $\mathbf{w}=(w_1,\dots,w_n)$ be a random vector taking values in $\R^n$. The \textbf{cumulants}~$\kappa_m$ of $\mathbf{w}$ are defined as the Taylor coefficients of the log-characteristic function of~$\mathbf{w}$, i.e.,
\begin{displaymath}
\ln\E[\re^{i\mathbf{t}\cdot\mathbf{w}}]=\sum_m\kappa_m\frac{(i\mathbf{t})^m}{m!},
\end{displaymath}
where the sum is taken over all multi-indices $m=(m_1,\dots,m_n)\in\N^n$ and $m!=\smash{\prod_{j=1}^n(m_j!)}$. For a multiset $B\subset\{1,\dots,n\}$ with~${|B|=k}$, we also write~$\kap{k}(w_j|j\in B)$ instead of $\kappa_m$, where $m_i$ is the multiplicity of $i\in B$.
\end{definition}

The cumulants of $\mathbf{w}$ satisfy the moment-cumulant relation
\begin{equation}
\E(w_1\dots w_n)=\sum_{\pi\in\Pi_n}\prod_{B\in\pi}\kap{|B|}(w_j|j\in B)\label{moment-cumulant-relation},
\end{equation}
where $\Pi_n$ denotes the set of partitions of $\{1,\dots,n\}$ (see, e.g.,~\cite{MingoSpeicher2017}). The complex cumulants arising whenever $\mathbf{w}$ takes values in $\C^n$ can be reduced to the above definition by considering the real and imaginary part of the random variables separately. In particular, the complex~$m$-cumulant of $w_{j_1},\dots,w_{j_m}$ is uniquely determined as a linear combination of the $m$-cumulants of the $2m$ random variables $\rRe(w_{j_1}),\rIm(w_{j_1}),\dots,\rRe(w_{j_m}),\rIm(w_{j_m})$. Writing $w_j=x_j+iy_j$ with $j=1,2$, we have,~{e.g.,}
\begin{displaymath}
\kap{2}(w_1,w_2)=\kap{2}(x_1,x_2)+i\kap{2}(x_1,y_2)-i\kap{2}(y_1,x_2)-\kap{2}(y_1,y_2).
\end{displaymath}
To keep the notation short, we usually view the cumulants as a function of the indices of the matrix elements by identifying~$\kap{k}(W_{a_1,a_2},\dots)$ with $\kap{k}(a_1a_2,\dots)$ or~$\kap{k}(\alpha_1,\dots,\alpha_k)$ using $\alpha_1,\dots,\alpha_k\in\{1,\dots,N\}^2$. Further, $d$ denotes the Euclidean distance on $\{1,\dots,N\}^2$ modulo the (Hermitian) symmetry,~i.e.,
\begin{displaymath}
d(a_1a_2,a_3a_4):=\min\{|a_1-a_3|+|a_2-a_4|,|a_1-a_4|+|a_2-a_3|\}.
\end{displaymath}
In this notation, the conditions from the polynomially decaying metric correlation structure can be formulated as follows.
\begin{assumption}[A3]\labelA{A3}
Whenever $W$ is a real random matrix, the $k$-cumulants~$\kap{k}$ of the elements of $W$ satisfy
\begin{align}
|\kap{2}(a_1a_2,a_3a_4)|&\leq\frac{C_{\kappa,2}}{1+d(a_1a_2,a_3a_4)^s},\label{A3-2-cumu}\\
|\kap{k}(\alpha_1,\dots,\alpha_k)|&\leq C_{\kappa,k}\prod_{e\in T_{min}}|\kap{2}(e)|,\ k\geq3,\label{A3-k-cumu}
\end{align}
for $s>2$ and some constants $C_{\kappa,k}>0$ for $k\geq2$. Here, $T_{min}$ is the (any) minimal spanning tree\footnote{A spanning tree is a graph that connects all the vertices without cycles. Any spanning tree with the lowest possible sum of edge weights is called minimal.} on the complete graph on $k$ vertices labelled by $\alpha_1,\dots,\alpha_k$ with edge weights $d(\alpha_i,\alpha_j)$. If~$W$ is a complex random matrix, we assume the above conditions for the real and imaginary part of $W$ separately in the sense that $\kap{k}(M_{\alpha_1},\dots,M_{\alpha_k})$ satisfies~\eqref{A3-2-cumu} or~\eqref{A3-k-cumu}, respectively, for any combination of $M_{\alpha}\in\{\rRe(W_{\alpha}),\rIm(W_{\alpha})\}$.
\end{assumption}
Note that a correlation decay of the form~\eqref{A3-k-cumu} arises in different statistical physics models~(see~\cite{DuneauIagolnitzerSouillard1973}).

\vspace{-2mm}
\subsection{Statement of the Main Result} 
With the notation established, we give the statement on the operator norm of $H$ as follows.
\begin{theorem}\label{main}
Under the assumptions \eqref{A1}-\eqref{A3}, we have that for all~${\epsilon>0}$,~${D>0}$ there exists a suitable constant~$C(\epsilon,D)$ such that, for all $N\in\N$,
\begin{equation*}
\P\big(\|H\|>N^{\epsilon}\big)\leq C(\epsilon,D)N^{-D}.
\end{equation*}
\end{theorem}

\begin{remark}
In~\cite{ErdosKruegerSchroeder2017}, the bound of Theorem~\ref{main} (for $s>12$) is used as a priori control for the operator norm in proving optimal local laws for random matrices with slow correlation decay. As a consequence of these results, the eigenvalues of $H$ are asymptotically confined to the support of the self-consistent density of states with high probability. This improves upon the above stochastic domination bound by replacing the $N^\epsilon$ factor with a large constant~(see~\cite[Cor.~2.3]{ErdosKruegerSchroeder2017}).
\end{remark}

\vspace{-2mm}
\subsection{Setup for the Moment Method}
We show that the assumptions \eqref{A1}-\eqref{A3} imply that
\begin{equation}\label{est-by-const}
\E[\tfrac{1}{N}\tr(H^k)]\leq C(k),\quad \forall k\in\N,
\end{equation}
from which the statement of Theorem~\ref{main} follows by an application of Chebyshev's inequality. As $H$ is Hermitian, we have $\|H\|^k\leq\tr(H^k)$ for all even $k\in\N$. This implies
\begin{align*}
\P(\|H\|>N^{\epsilon})\leq\frac{\E[\|H\|^k]}{N^{k\epsilon}}\leq\frac{N\E[\tfrac{1}{N}\tr(H^k)]}{N^{k\epsilon}}\leq\frac{NC(k)}{N^{k\epsilon}}
\end{align*}
and thus gives the desired bound if $k$ is chosen large enough.

\medskip
Expanding the term on the left-hand side of~\eqref{est-by-const} using the relation~\eqref{moment-cumulant-relation} yields
\begin{align}
\E[\tfrac{1}{N}\tr(H^k)]&=\frac{1}{N}\sum_{a_1,\dots,a_k}\E[H_{a_1,a_2}H_{a_2,a_3}\dots H_{a_k,a_1}]\nonumber\\
&=N^{-k/2-1}\sum_{\pi\in\Pi_k}\sum_{a_1,\dots,a_k}\prod_{B\in\pi}\kap{|B|}(a_ja_{j+1}|j\in B),\label{cumu-expansion}
\end{align}
where $\Pi_k$ denotes the set of partitions of $\{1,\dots,k\}$ and the index $j+1$ is to be interpreted $\mathrm{mod}\ k$, i.e., if $j=k$, then $a_ja_{j+1}=a_ka_1$. Observe that all terms involving~1-cumulants vanish due to \eqref{A1}. Hence, one can restrict the sum to partitions $\pi$ without singleton sets. The cumulant expansion~\eqref{cumu-expansion} is the main difference between the real symmetric and complex Hermitian case, as considering $H\in\C^{N\times N}$ requires replacing the cumulants by their complex counterparts. However, one can always reduce the argument to the real case by considering the real and imaginary parts of the random variables separately. Thus, from now on we assume $H$ to be real. We develop the estimates by first deriving suitable bounds for the products that only involve 2-cumulants, which include the leading terms, and then successively incorporating higher-order cumulants. The fast correlation decay from Assumption~\eqref{A3} implies that roughly half of the summations over the indices~$a_1,\dots,a_k$ yield a factor $N$, while the other half can be summed up with an~$N$-independent bound. To quantify this behavior, we introduce the following counting rule, which constitutes a useful property of a term involving $n\in\N_0$ summations.
\begin{crule}[CR]\labelR{CR}
A  term involving several index summations and a product of cumulants satisfies the counting rule if every independent summation over an index $a_{i_1},\dots,a_{i_n}$ yields a contribution of order~$\sqrt{N}$.
\end{crule}
Whenever one can show that~\eqref{CR} holds for a given term, the counting rule allows to replace explicit bounds by cruder power counting arguments, which simplifies calculating the contribution to~\eqref{cumu-expansion}. Throughout the proof of~\eqref{est-by-const}, we, therefore, aim to show this property for as many terms as possible. Note, however, that bounding the leading terms requires an extra power of~$N$ compared to the counting rule (see the proof of Lemma~\ref{2-cumu-lemma} below). For these terms, the factor~$N^{-k/2-1}$ in~\eqref{cumu-expansion} has to be canceled out completely.

\vspace{-2mm}
\subsection{Visualization Using Graphs}\label{sect-graphs}
Following the discussion of the previous subsection, we aim to bound
\begin{equation}\label{prototype-cumu-term}
N^{-k/2-1}\Big|\sum_{a_1,\dots,a_k}\prod_{B\in\pi}\kap{|B|}(a_ja_{j+1}|j\in B)\Big|
\end{equation}
for any given $\pi\in\Pi_k$. Terms of this form can be visualized by considering a $k$-gon, where the vertices are labeled by the indices $a_1,\dots,a_k$ and the edges by the successive double indices~$(a_1,a_2),\dots,(a_k,a_1)$. In this picture, every $j$-cumulant combines $j$ distinct edges such that every edge appears in exactly one of the cumulants in the product. We say that a vertex \emph{belongs to a~j-cumulant}, if it is adjacent to one of the edges associated with it. A vertex~(resp. the corresponding index) which only occurs in a single cumulant is referred to as \emph{internal vertex}~(resp. \emph{internal index}). We define the corresponding graphs as follows.

\begin{definition}
The graph given by the regular $k$-gon with vertices $a_1,\dots,a_k$ is denoted as $\Gamma_k$. For any $\pi\in\Pi_k$, we define $\Gamma_k(\pi)$ as the colored graph in which the $j$ edges that belong to the same $j$-cumulant are assigned the same color.
\end{definition}

Examples of this visualization are given in Fig.~\ref{fig1} and Fig.~\ref{fig2} below, where the colors are indicated by different linestyles. For later applications of algorithms on the graph, we further introduce an ordering. Assume first that $\pi$ is a pairing, i.e., $\Gamma_k(\pi)$ has exactly two edges of each color. Let~$E_k=\{(a_1,a_2),\dots,(a_k,a_1)\}$ denote the~(ordered) set of edges and~${\phi:E_k\rightarrow E_k}$ be the function that maps every edge to the other one of the same color. Starting with~$e_1=(a_1,a_2)$, we go through the elements of $E_k$ and note the pairing~$(e_n,\phi(e_n))$ whenever $e_n\neq\phi(e_i)$ for all~$i<n$ to obtain~$k/2$ pairs of edges. We denote the pairings as an~ordered set~${C(\pi):=\{(e_1,\phi(e_1)),\dots,(e_{k/2},\phi(e_{k/2}))\}}$ and refer to the edges in $C(\pi)$ as \emph{paired edges}. Similarly, we can extract a set $C(\pi)$ of~$j$-tuples,~$j\geq 2$, for any partition~$\pi\in\Pi_k$.

\begin{definition}
For a partition $\pi\in\Pi_k$, we denote the graph with the (ordered) pairing structure induced by $\pi$ by the pair~$(\Gamma_k,C(\pi))$.
\end{definition}

Note that the above construction induces a one-to-one correspondence between cumulant terms of the form~\eqref{prototype-cumu-term} and colored $k$-gons. In particular, both $\Gamma_k(\pi)$ and $(\Gamma_k,C(\pi))$ contain all information on the corresponding term in~\eqref{prototype-cumu-term}.

\vspace{-2mm}
\section{Proof of~\eqref{est-by-const} for Terms Involving Only 2-Cumulants}\label{sect-2-cumu-estimates}
In this section, we focus on the following special case.
\begin{lemma}\label{2-cumu-lemma}
Let $k\in\N$ be even and $\pi\in\Pi_k$ with $|B|=2$ for all $B\in\pi$. Then,
\begin{equation}\label{2-cumu-term}
N^{-k/2-1}\Big|\sum_{a_1,\dots,a_k}\prod_{B\in\pi}\kap{2}(a_ja_{j+1}|j\in B)\Big|\leq C(k).
\end{equation}
\end{lemma}
Assume first that $\pi$ is chosen such that the 2-cumulants occurring in the term do not involve internal indices. We note the following general estimates whose proofs are elementary from~\eqref{A3-2-cumu} and the fact that $s>2$.

\begin{lemma}\label{red-rule-2-cumu}
Assume that \eqref{A3} holds. Then
\begin{align*}
\sum_{a_1}|\kap{2}(a_1a_2,a_3a_4)|&\leq C,\quad \sum_{a_1,a_2}|\kap{2}(a_1a_2,a_3a_4)|\leq C,\quad  \sum_{a_1,a_3}|\kap{2}(a_1a_2,a_3a_4)|\leq CN,\\
\sum_{a_1,a_2,a_3}|\kap{2}(a_1a_2,a_3a_4)|&\leq CN,\quad \sum_{a_1,\dots,a_4}|\kap{2}(a_1a_2,a_3a_4)|\leq CN^2
\end{align*}
uniformly for any choice of the unsummed indices.
\end{lemma}

Note that all estimates in Lemma~\ref{red-rule-2-cumu} follow~\eqref{CR}, some of them are even stronger. However, the summation over an internal index would, in general, not obey this counting rule, since~$\smash{\sum_{a_2=1}^N}|\kap{2}(a_1a_2,a_2a_3)|$ may be of order $N$ if $a_1=a_3$. This is the reason why internal indices are treated separately. The key to estimating~\eqref{2-cumu-term} in the given case is a recursive summation procedure. We demonstrate the approach for the term visualized in Fig.~\ref{fig1} below, where different linestyles indicate the edges associated with the same 2-cumulant.

\begin{center}
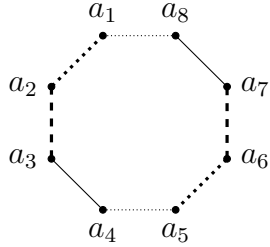

\begin{tikzpicture}[scale=1.25]
\draw (-0.3827,0.9239) node[above=1pt] {$a_1$};
\draw (-0.9239,0.3827) node[left=1pt] {$a_2$};
\draw (-0.9239,-0.3827) node[left=1pt] {$a_3$};
\draw (-0.3827,-0.9239) node[below=1pt] {$a_4$};
\draw (0.3827,-0.9239) node[below=1pt] {$a_5$};
\draw (0.9239,-0.3827) node[right=1pt] {$a_6$};
\draw (0.9239,0.3827) node[right=1pt] {$a_7$};
\draw (0.3827,0.9239) node[above=1pt] {$a_8$};
\draw[black,dotted,very thick] (-0.3827,0.9239) -- (-0.9239,0.3827);
\draw[black,dashed,very thick] (-0.9239,0.3827) -- (-0.9239,-0.3827);
\draw[black] (-0.9239,-0.3827) -- (-0.3827,-0.9239);
\draw[black,densely dotted] (-0.3827,-0.9239) -- (0.3827,-0.9239);
\draw[black,dotted,very thick] (0.3827,-0.9239) -- (0.9239,-0.3827);
\draw[black,dashed,very thick] (0.9239,-0.3827) -- (0.9239,0.3827);
\draw[black] (0.9239,0.3827) -- (0.3827,0.9239);
\draw[black,densely dotted] (0.3827,0.9239) -- (-0.3827,0.9239);
\filldraw [black] (-0.3827,0.9239) circle (1pt)(-0.9239,0.3827) circle (1pt)(-0.9239,-0.3827) circle (1pt)(-0.3827,-0.9239) circle (1pt);
\filldraw [black] (0.3827,-0.9239) circle (1pt)(0.9239,-0.3827) circle (1pt)(0.9239,0.3827) circle (1pt)(0.3827,0.9239) circle (1pt);
\end{tikzpicture}
\captionof{figure}{Visualization of $\kap{2}(a_1a_2,a_5a_6)\kap{2}(a_2a_3,a_6a_7)\kap{2}(a_3a_4,a_7a_8)\kap{2}(a_4a_5,a_8a_1)$}\label{fig1}
\end{center}
\begin{example}\label{ex-nointindices}
To start the summation, estimate
\begin{align}
S&:=N^{-5}\sum_{a_1,\dots,a_8}|\kap{2}(a_1a_2,a_5a_6)\kap{2}(a_2a_3,a_6a_7)\kap{2}(a_3a_4,a_7a_8)\kap{2}(a_4a_5,a_8a_1)|\nonumber\\
&\leq N^{-5}\sum_{a_1,\dots,a_8,a_1'}|\kap{2}(a_1a_2,a_5a_6)\kap{2}(a_2a_3,a_6a_7)\kap{2}(a_3a_4,a_7a_8)\kap{2}(a_4a_5,a_8a_1')|.\label{start-sum}
\end{align}
Adding the extra summation label may $a_1'$ appear as an unnecessary overestimate, but it gives a dedicated start and end point for the summation algorithm by breaking the cyclic structure of the graph. Next, isolate the~2-cumulant involving~$a_1$ by taking the maximum over the remaining indices~$a_2,a_5,a_6$ for the other factors. Together with $a_1$, we sum over all labels appearing in $\kap{2}(a_1a_2,a_5a_6)$. Note that, by Lemma~\ref{red-rule-2-cumu}, the counting rule (CR) applies to this sum, giving a factor of~$(\sqrt{N})^4=N^2$ and the estimate
\begin{align*}
S\leq N^{-5}CN^2\max_{a_2,a_5,a_6}\Big(\sum_{a_3,a_4,a_7,a_8,a_1'}|\kap{2}(a_2a_3,a_6a_7)\kap{2}(a_3a_4,a_7a_8)\kap{2}(a_4a_5,a_8a_1')|\Big).
\end{align*}
From $a_1$, continue counter-clockwise along the octagon to find the next index to sum over, i.e., $a_3$. Identifying and isolating $\kap{2}(a_2a_3,a_6a_7)$, sum over all remaining indices in the factor, i.e., $a_3$ and $a_7$. Again, Lemma~\ref{red-rule-2-cumu} justifies the use of~\eqref{CR}, giving a contribution of $(\sqrt{N})^2=N$ and the estimate
\begin{align*}
S\leq CN^{-2}\max_{a_3,a_5,a_7}\Big(\sum_{a_4,a_8,a_1'}|\kap{2}(a_3a_4,a_7a_8)\kap{2}(a_4a_5,a_8a_1')|\Big).
\end{align*}
Repeating the previous step, continuing along the octagon yields~$a_4$ as the next index. Isolating $\kap{2}(a_3a_4,a_7a_8)$, we apply Lemma~\ref{red-rule-2-cumu} to perform the summations over $a_4$ and $a_8$ using~\eqref{CR}. This yields a factor of ${(\sqrt{N})^2=N}$ and
\begin{align*}
S\leq CN^{-1}\max_{a_4,a_5,a_8}\Big(\sum_{a_1'}|\kap{2}(a_4a_5,a_8a_1')|\Big).
\end{align*}
Lastly, the summation over $a_{1'}$ in $|\kap{2}(a_4a_5,a_8a_1')|$ remains. By applying the counting rule again, we obtain a final bound of order $N^{-1/2}$, which can be estimated by a constant as claimed in Lemma~\ref{2-cumu-lemma}. In particular, the pairing in Fig. 1 gives a subleading contribution to~\eqref{cumu-expansion}. Note that introducing an additional index in the first step is affordable, as are any overestimates from using~\eqref{CR} over the explicit estimates in Lemma~\ref{red-rule-2-cumu}. The cruder bounds allow to carry out every step but the initial one following the same pattern, which reduces the final bound to a counting argument.
\end{example}

As this recursive summation procedure relies on Lemma~\ref{red-rule-2-cumu}, it cannot be applied directly if summation over internal indices are present. To prepare for the general case, we extend Lemma~\ref{red-rule-2-cumu} to more general objects.

\begin{definition}
For ${\mathbf{x}\in\R^N}$, set
\begin{displaymath}
{\kap{2}(\mathbf{x}a_2,a_3a_4):=\sum_{a_1=1}^N\kap{2}(a_1a_2,a_3a_4)x_{a_1}}
\end{displaymath}
and define any 2-cumulants with one or more indices replaced by a vector analogously.
\end{definition}

We collect some estimates for 2-cumulants of this form below, where we follow the convention that vectors only occur in place of the first index of an index pair and only replace unsummed indices. However, the same bounds hold if the second index of the respective pair is replaced instead. As in Lemma~\ref{red-rule-2-cumu}, no internal index is summed~up.

\begin{lemma}\label{ext-red-rule-2-cumu}
Assume that \eqref{A3} holds and let $\mathbf{x},\mathbf{y}\in\R^N$. Then we have
\begin{align*}
\sum_{a_2}|\kap{2}(\mathbf{x}a_2,a_3a_4)|&\leq C\|\mathbf{x}\|_2,\quad \sum_{a_3}|\kap{2}(\mathbf{x}a_2,a_3a_4)|\leq CN^{1/2}\|\mathbf{x}\|_2,\\
 \sum_{a_2,a_3}|\kap{2}(\mathbf{x}a_2,a_3a_4)|&\leq CN\|\mathbf{x}\|_2,\quad \sum_{a_3,a_4}|\kap{2}(\mathbf{x}a_2,a_3a_4)|\leq CN\|\mathbf{x}\|_2,\\ \sum_{a_2,a_3,a_4}|\kap{2}(\mathbf{x}a_2,a_3a_4)|&\leq CN^{3/2}\|\mathbf{x}\|_2
\end{align*}
uniformly for any choice of unsummed indices. Similar bounds hold with two vectors, i.e.,
\begin{align*}
|\kap{2}(\mathbf{x}a_2,\mathbf{y}a_4)|&\leq C\|\mathbf{x}\|_2\|\mathbf{y}\|_2,\quad \sum_{a_2}|\kap{2}(\mathbf{x}a_2,\mathbf{y}a_4)|\leq CN^{1/2}\|\mathbf{x}\|_2\|\mathbf{y}\|_2,\\
\sum_{a_2,a_4}|\kap{2}(\mathbf{x}a_2,\mathbf{y}a_4)|&\leq CN\|\mathbf{x}\|_2\|\mathbf{y}\|_2.
\end{align*}
In particular, the estimates follow~\eqref{CR}.
\end{lemma}

\begin{proof}
Noting that $\max_j|x_j|\leq\|\mathbf{x}\|_2$, the bounds in Lemma~\ref{red-rule-2-cumu} imply that
\begin{align*}
\sum_{a_2}|\kap{2}(\mathbf{x}a_2,a_3a_4)|\leq \|\mathbf{x}\|_2 \sum_{a_1,a_2}|\kap{2}(a_1a_2,a_3a_4)|\leq C\|\mathbf{x}\|_2.
\end{align*}
If the summation over $a_2$ is replaced by a summation over $a_3$ or $a_4$, we obtain a bound of order $N$ instead. Further, an application of the Cauchy-Schwarz inequality leads to
\begin{align}
\sum_{a_3}|\kap{2}(\mathbf{x}a_2,a_3a_4)|&\leq\|\mathbf{x}\|_2\sqrt{\sum_{a_1}\Big(\sum_{a_3}|\kap{2}(a_1a_2,a_3a_4)|\Big)^2}\leq CN^{1/2}\|\mathbf{x}\|_2.\label{1sum-1vector}
\end{align}
Recalling that the summation over three indices gives a factor $N$, the bound for summation over all three indices in $|\kap{2}(\mathbf{x}a_2,a_3a_4)|$ follows analogously.

\medskip
For 2-cumulants that involve two vectors, applying~\eqref{A3-2-cumu} yields
\begin{align}
|\kap{2}(\mathbf{x}a_2,\mathbf{y}a_4)|&\leq\sum_{a_1,a_3}\Big(\frac{ C_{\kappa}|x_{a_1}y_{a_3}|}{1+|a_1-a_3|^s+|a_2-a_4|^s}+\frac{ C_{\kappa}|x_{a_1}y_{a_3}|}{1+|a_1-a_4|^s+|a_2-a_3|^s}\Big).\label{0sums-2vectors}
\end{align}
Set~${\varepsilon=\|\mathbf{y}\|_2/\|\mathbf{x}\|_2}$ and estimate the first term as
\begin{align*}
\sum_{a_1,a_3}\frac{|x_{a_1}y_{a_3}|}{1+|a_1-a_3|^s+|a_2-a_4|^s}&\leq \sum_{a_1,a_3}\frac{\epsilon|x_{a_1}|^2+\epsilon^{-1}|y_{a_3}|^2}{1+|a_1-a_3|^s}\leq C(\epsilon\|\mathbf{x}\|_2^2+\epsilon^{-1}\|\mathbf{y}\|_2^2)
\end{align*}
to obtain an $N$-independent bound. The estimate of the second term is similar.

\medskip
Adding a summation over one index, e.g., $a_2$, two applications of the Cauchy-Schwarz inequality and the third estimate of Lemma~\ref{red-rule-2-cumu} yield
\begin{align}
\sum_{a_2}|\kap{2}(\mathbf{x}a_2,\mathbf{y}a_4)|&\leq\|\mathbf{x}\|_2\|\mathbf{y}\|_2\sqrt{\sum_{a_3}\Big(\sum_{a_1}\sum_{a_2}|\kap{2}(a_1a_2,a_3a_4)|\Big)^2}\leq CN^{1/2}\|\mathbf{x}\|_2\|\mathbf{y}\|_2.\label{1sum-2vectors}
\end{align}
Finally, it follows
\begin{align*}
&\sum_{a_2,a_4}|\kap{2}(\mathbf{x}a_2,\mathbf{y}a_4)|\leq \sum_{a_1,\dots,a_4}\Big(\frac{C_{\kappa}|x_{a_1}y_{a_3}|}{1+|a_1-a_3|^s+|a_2-a_4|^s}+\frac{C_{\kappa}|x_{a_1}y_{a_3}|}{1+|a_1-a_4|^s+|a_2-a_3|^s}\Big).
\end{align*}
Here, we obtain
\begin{align*}
\sum_{a_1,\dots,a_4=1}^N\frac{|x_{a_1}y_{a_3}|}{1+|a_1-a_3|^s+|a_2-a_4|^s}&\leq CN\sum_{a_1,a_3=1}^N\frac{|x_{a_1}y_{a_3}|}{1+|a_1-a_3|^{s-1}}\leq CN\|\mathbf{x}\|_2\|\mathbf{y}\|_2,
\end{align*}
and the estimate for the second term is similar.
\end{proof}

Lemma~\ref{ext-red-rule-2-cumu} is the key tool for estimating terms that include summation over internal indices. Consider the matrix~${T\in\R^{N\times N}}$ defined by its matrix elements
\begin{equation}\label{def-T}
T_{a_1,a_3}:=\sum_{a_2}T^{(a_2)}_{a_1,a_3}:=\sum_{a_2}\kap{2}(a_1a_2,a_2a_3)
\end{equation}
and observe that $|T_{a_1,a_3}|\leq C_{\kappa}N$ by~\eqref{A3-2-cumu}, but also $\|T\|\leq CN$, since
\begin{align*}
\|T^{(a_2)}\mathbf{x}\|_2^2&\leq\sum_{a_1}\Big(\sum_{a_3}|\kap{2}(a_1a_2,a_2a_3)x_{a_3}|\Big)^2\\
&\leq C\sum_{a_3,a_3'}\sum_{a_1}\frac{1}{1+|a_1-a_3|^s} \frac{1}{1+|a_1-a_3'|^s}(x_{a_3}^2+x_{a_3'}^2)\\
&\leq C\|\mathbf{x}\|_2^2
\end{align*}
for $\mathbf{x}\in\R^N$, and $\|T\|\leq\sum_{a_2}\|T^{(a_2)}\|$. Next, we derive similar estimates for the matrix~${T^{[j]}\in\R^{N\times N}}$ defined for ${2\leq j\leq k-1}$ by
\begin{equation}\label{def-B}
T^{[j]}_{a_1,a_{2j+1}}:=\sum_{a_2,a_{2j}}\kap{2}(a_1a_2,a_{2j}a_{2j+1})T_{a_2,a_{2j}}^{j-1}.
\end{equation}
Note that the superscript corresponds to the total number of 2-cumulants that are rewritten to obtain $T^{[j]}_{a_1,a_{2j+1}}$. Again, we have $\smash{|T^{[j]}_{a_1,a_{2j+1}}|}\leq CN^j$ by a direct estimate, but also
\begin{equation}\label{norm-Tj}
\|T^{[j]}\|\leq CN^j\quad \forall j\in\{2,\dots, k-1\}.
\end{equation}
As the argument is the same in the general case, we consider only $j=2$. Define the vector~$\mathbf{y}^{a_4}$ through $(\mathbf{y}^{a_4})_{a_2}:=N^{-1}T_{a_2,a_4}$, $a_2=1,\dots,N$ and let
\begin{displaymath}
T^{[2],a_4}:=\sum_{a_2}\kap{2}(a_1a_2,a_4a_5)T_{a_2,a_4}.
\end{displaymath}
Observing that~${\|\mathbf{y}^{a_4}\|_2\leq C}$ uniformly in~$a_4$, we obtain
\begin{align}
\|T^{[2],a_4}\mathbf{x}\|_2^2&\leq N^2\sum_{a_1}\Big(\sum_{a_2,a_5}|\kap{2}(a_1a_2,a_4a_5)(\mathbf{y}^{a_4})_{a_2}x_{a_5}|\Big)^2\leq CN^2\label{norm-B}
\end{align}
for any $\mathbf{x}\in\R^N$ with $\|\mathbf{x}\|_2\leq 1$, which implies $\|T^{[2]}\|\leq N\max_{a_4}\|T^{[2],a_4}\|\leq CN^2$. The second inequality of~\eqref{norm-B} follows from applying~\eqref{A3-2-cumu} and estimating the $a_1$ summation using a suitable distinction of cases. Note that the structure of the matrix-vector multiplication in $T^{[2],a_4}\mathbf{x}$ yields entries of the form
\begin{displaymath}
\sum_{a_2,a_5}\kap{2}(a_1a_2,a_4a_5)(\mathbf{y}^{a_4})_{a_2}x_{a_5}=\kap{2}(a_1\mathbf{y}^{a_4},a_4\mathbf{x}),
\end{displaymath}
 which do not allow for the usual convention of the vector occurring only as the first index of an index pair.

\medskip
We now demonstrate the approach for treating general products of 2-cumulants for the terms visualized in Fig. \ref{fig2} below.

\begin{figure}[H]
\begin{center}
\begin{tikzpicture}[scale=1.25] 
\draw (0,1) node[above=1pt] {$a_1$};
\draw (-0.5878,0.809) node[above left=1pt] {$a_2$};
\draw (-0.951,0.309) node[left=1pt] {$a_3$};
\draw (-0.951,-0.309) node[left=1pt] {$a_4$};
\draw (-0.5878,-0.809) node[below left=1pt] {$a_5$};
\draw (0,-1) node[below=1pt] {$a_6$};
\draw (0.5878,-0.809) node[below right=1pt] {$a_7$};
\draw (0.951,-0.309) node[right=1pt] {$a_8$};
\draw (0.951,0.309) node[right=1pt] {$a_9$};
\draw (0.5878,0.809) node[above right=1pt] {$a_{10}$};
\draw[black] (0,1) -- (-0.5878,0.809);
\draw[black] (-0.5878,0.809) -- (-0.951,0.309);
\draw[black, densely dotted] (-0.951,0.309) -- (-0.951,-0.309);
\draw[black, densely dotted] (-0.951,-0.309) -- (-0.5878,-0.809);
\draw[black, dashdotted] (-0.5878,-0.809) -- (0,-1);
\draw[black,dashed,very thick] (0,-1) -- (0.5878,-0.809);
\draw[black,dotted,very thick] (0.5878,-0.809) -- (0.951,-0.309);
\draw[black,dotted,very thick] (0.951,-0.309) -- (0.951,0.309);
\draw[black, dashdotted] (0.951,0.309) -- (0.5878,0.809);
\draw[black,dashed,very thick] (0.5878,0.809) -- (0,1);
\filldraw [black] (0,1) circle (1pt)(-0.5878,0.809) circle (1pt)(-0.951,0.309) circle (1pt)(-0.951,-0.309) circle (1pt)(-0.5878,-0.809) circle (1pt);
\filldraw [black] (0,-1) circle (1pt)(0.5878,-0.809) circle (1pt)(0.951,-0.309) circle (1pt)(0.951,0.309) circle (1pt)(0.5878,0.809) circle (1pt);\end{tikzpicture}\hspace{2cm}
\begin{tikzpicture}[scale=1.25] 
\draw (0,1) node[above=1pt] {$a_1$};
\draw (-0.5878,0.809) node[above left=1pt] {$a_2$};
\draw (-0.951,0.309) node[left=1pt] {$a_3$};
\draw (-0.951,-0.309) node[left=1pt] {$a_4$};
\draw (-0.5878,-0.809) node[below left=1pt] {$a_5$};
\draw (0,-1) node[below=1pt] {$a_6$};
\draw (0.5878,-0.809) node[below right=1pt] {$a_7$};
\draw (0.951,-0.309) node[right=1pt] {$a_8$};
\draw (0.951,0.309) node[right=1pt] {$a_9$};
\draw (0.5878,0.809) node[above right=1pt] {$a_{10}$};
\draw[black] (0,1) -- (-0.5878,0.809);
\draw[black] (-0.5878,0.809) -- (-0.951,0.309);
\draw[black, densely dotted] (-0.951,0.309) -- (-0.951,-0.309);
\draw[black, densely dotted] (-0.951,-0.309) -- (-0.5878,-0.809);
\draw[black, dashdotted] (-0.5878,-0.809) -- (0,-1);
\draw[black,dashed,very thick] (0,-1) -- (0.5878,-0.809);
\draw[black,dotted,very thick] (0.5878,-0.809) -- (0.951,-0.309);
\draw[black,dotted,very thick] (0.951,-0.309) -- (0.951,0.309);
\draw[black,dashed,very thick] (0.951,0.309) -- (0.5878,0.809);
\draw[black, dashdotted] (0.5878,0.809) -- (0,1);
\filldraw [black] (0,1) circle (1pt)(-0.5878,0.809) circle (1pt)(-0.951,0.309) circle (1pt)(-0.951,-0.309) circle (1pt)(-0.5878,-0.809) circle (1pt);
\filldraw [black] (0,-1) circle (1pt)(0.5878,-0.809) circle (1pt)(0.951,-0.309) circle (1pt)(0.951,0.309) circle (1pt)(0.5878,0.809) circle (1pt);\end{tikzpicture}
\caption{A crossing (left) and a non-crossing pairing (right) for $k=10$.}\label{fig2}
\end{center}
\end{figure}
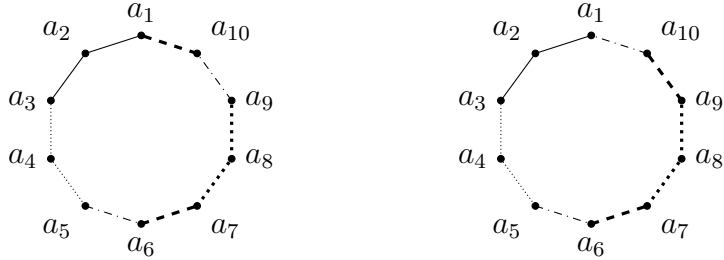
\vspace{-1cm}
\begin{example}\label{ex-intindices}
First, consider the term on the left of Fig.~2. Recalling the definition of the matrix $T$ from~\eqref{def-T}, rewrite the summation over the internal indices $a_2$, $a_4$ and~$a_7$~as
\begin{align}
&N^{-6}\sum_{a_1,\dots,a_{10}}\kap{2}(a_1a_2,a_2a_3)\kap{2}(a_3a_4,a_4a_5)\kap{2}(a_5a_6,a_9a_{10})\kap{2}(a_6a_7,a_{10}a_1)\kap{2}(a_7a_8,a_8a_9)\nonumber\\
&=N^{-6}\sum_{a_1,a_5,a_6,a_7,a_9,a_{10}}T^2_{a_1,a_5}\kap{2}(a_5a_6,a_9a_{10})\kap{2}(a_6a_7,a_{10}a_1)T_{a_7,a_9}\nonumber\\
&=N^{-3}\sum_{a_1,a_6,a_7,a_{10}}\kap{2}(\mathbf{x}^{a_1}a_6,\mathbf{y}^{a_7}a_{10})\kap{2}(a_6a_7,a_{10}a_1),\label{rewritten-ex}
\end{align}
where the vectors $\mathbf{x}^{a_1},\mathbf{y}^{a_7}\in\R^N$ in the last step are given by
\begin{align*}
x^{a_1}_{a_5}&=\frac{1}{N^2}T^2_{a_1,a_5},\ a_5=1,\dots,N,\quad y^{a_7}_{a_9}=\frac{1}{N}T_{a_7,a_9},\ a_9=1,\dots,N.
\end{align*}
To keep the notation consistent with the proof in the general case and Lemma~\ref{ext-red-rule-2-cumu}, we introduce the convention that vectors obtained from matrix elements are always defined via the rows of the respective matrix. Recalling that~$\|T\|\leq CN$, we have $\|\mathbf{x}^{a_1}\|_2,\|\mathbf{y}^{a_7}\|_2\leq C$ uniformly for any choice of $a_1$ and $a_7$, respectively. Note that one factor of $N$ per power of~$T$ is written in front of the sum and that the convention chosen for the vectors ensures that we always replace the first index of an index pair. Further, the sum obtained from~\eqref{rewritten-ex} does not involve summation over internal indices. Modifying the recursive summation procedure from Example~\ref{ex-nointindices} by also taking the maximum over $a_1$ in the 2-cumulant involving $\mathbf{x}^{a_1}$ instead of introducing the additional summation label $a_1'$, and using the bounds from Lemmas~\ref{red-rule-2-cumu} and~\ref{ext-red-rule-2-cumu} yields
\begin{align}
&N^{-3}\sum_{a_1,a_6,a_7,a_{10}}|\kap{2}(\mathbf{x}^{a_1}a_6,\mathbf{y}^{a_7}a_{10})\kap{2}(a_6a_7,a_{10}a_1)|\nonumber\\
&\leq N^{-3} \max_{a_1,a_6,a_7,a_{10}}|\kap{2}(\mathbf{x}^{a_1}a_6,\mathbf{y}^{a_7}a_{10})|\sum_{a_1,a_6,a_7,a_{10}}|\kap{2}(a_6a_7,a_{10}a_1)|\nonumber\\
&\leq CN^{-3}N^2\max_{a_1,a_6,a_7,a_{10}}|\kap{2}(\mathbf{x}^{a_1}a_6,\mathbf{y}^{a_7}a_{10})|\leq CN^{-1}\leq C,\label{eq-ex-withvectors}
\end{align}
showing that the pairing on the left of Fig. \ref{fig2} gives a sub-leading contribution to~\eqref{cumu-expansion}.

\medskip
In contrast, observe that the non-crossing pairing on the right of Fig.~\ref{fig2} needs to be handled differently, since treating $T$ as before yields~$\kap{2}(\mathbf{x}^{a_1}a_6,a_{10}a_1)$ and~$\kap{2}(a_6a_7,\mathbf{y}^{a_7}a_{10})$, which cannot be estimated using Lemma~\ref{ext-red-rule-2-cumu} due to the indices $a_1$ and~$a_7$ appearing twice in the respective 2-cumulants. However, the terms can be rewritten using the matrices~$T^{[j]}$ defined in~\eqref{def-B}. Recalling that $\|T^{[j]}\|\leq CN^j$ from~\eqref{norm-Tj}, we obtain
\begin{align*}
&N^{-6}\Big|\sum_{a_1,a_5,a_6,a_7,a_9,a_{10}}T^2_{a_1,a_5}\kap{2}(a_5a_6,a_{10}a_1)\kap{2}(a_6a_7,a_9a_{10})T_{a_7,a_9}\Big|\\
&=N^{-6}\Big|\sum_{a_6,a_{10}}T^{[3]}_{a_{10},a_6}T^{[2]}_{a_6,a_{10}}\Big|=N^{-6}\big|\tr \big(T^{[3]}T^{[2]}\big)\big|\leq CN^{-6}NN^3N^2=C.
\end{align*}
Hence, the term on the left of Fig.~\ref{fig2} yields a leading contribution to~\eqref{cumu-expansion}.
\end{example}

After all these preparations, we can give a complete proof of Lemma~\ref{2-cumu-lemma}.

\vspace{-2mm}
\subsection{Proof of Lemma~\ref{2-cumu-lemma}}
Let $k\geq2$ be even, $\pi\in\Pi_k$ such that $|B|=2$ for all $B\in\pi$. The proof of Lemma~\ref{2-cumu-lemma} is structured into two main steps. In the first step (Lemmas~\ref{rewriting-graph} and~\ref{rewriting-cumu}), we generalize the strategies from Example~\ref{ex-intindices} to rewrite the term on the left-hand side to a form that is tractable by a recursive summation procedure. We then carry out the required estimates in the second step~(Lemmas~\ref{estimates1} and~\ref{estimates2}), showing that the final bound is indeed of order one.

\medskip
\underline{\smash{Step 1: Rewriting}}\\
The aim of the rewriting step is to obtain a term of the form
\begin{equation}\label{rewrite-goal}
\sum_{a_1,\dots,a_k}\prod_{B\in\pi}\kap{2}(a_ja_{j+1}|j\in B)=\sum_{b_1,\dots,b_{k'}}\Big(\prod_{B\in\pi'}\kap{2}(b_jb_{j+1}|j\in B)\prod_{i\in J}M^{(i)}_{b_{i_1},b_{i_2}}\Big)
\end{equation}
where $k'\leq k$, $M^{(i)}\in\R^{N\times N}$, and $\pi'$ is chosen such that neither of the~2-cumulants on the right-hand side of~\eqref{rewrite-goal} involves internal indices and such that $b_{i_1},b_{i_2}$ always appear in some~2-cumulant, but never in the same one. Recall that this structure was also obtained in~\eqref{rewritten-ex} of Example~\ref{ex-intindices}. Whenever the left-hand side of~\eqref{2-cumu-term} does not involve summation over internal indices, choosing $k=k'$, $\pi=\pi'$, and $J=\emptyset$ in~\eqref{rewrite-goal} allows to skip the rewriting step~(see Example~\ref{ex-nointindices}).

\medskip
Following the discussion in Section~\ref{sect-graphs}, the term on the left-hand side of~\eqref{2-cumu-term} corresponds to the tuple $(\Gamma_k,C(\pi))$. We aim to visualize the right-hand side of~\eqref{rewrite-goal} as well by associating it with a polygon on the vertices $b_1,\dots,b_{k'}$. In this picture, any matrix element~$\smash{M^{(i)}_{b_{i_1},b_{i_2}}}$ corresponds to an edge connecting the vertices $b_{i_1},b_{i_2}$ and the two edges associated with the same 2-cumulant are assigned the same color. We further introduce edge weights $w$ to encode bounds of order $N^w$ on the corresponding matrices in the graph for the following estimates. The rewriting procedure to obtain~\eqref{rewrite-goal} thus corresponds to an algorithm on $(\Gamma_k,C(\pi))$, which we introduce in Lemma~\ref{rewriting-graph} below. An example is given in Fig.~\ref{fig3}.

\begin{lemma}\label{rewriting-graph}
Let $k\geq2$ be even, $\pi\in\Pi_k$ be a pairing and $(\Gamma_k,C(\pi))$ the corresponding graph introduced in Section~\ref{sect-graphs}. Then there exists an edge-weighted graph $\smash{\widetilde{\Gamma}_{k'}(\pi)}$ on $k'\leq k$ vertices with integer weights $w_i\geq0$ and an ordered set $\smash{\widetilde{C}(\pi)}$ of pairings such that the following properties~hold.
\begin{itemize}
\item[(a)] Edges that appear in $\smash{\widetilde{C}(\pi)}$ all have zero weight.
\item[(b)] Two edges with nonzero weight are never adjacent to each other.
\item[(c)] Two adjacent edges in $\smash{\widetilde{\Gamma}_{k'}(\pi)}$ that both appear in $\smash{\widetilde{C}(\pi)}$ belong to different pairs.
\item[(d)] Whenever two edges are adjacent to an edge assigned a nonzero weight, they appear in~$\smash{\widetilde{C}(\pi)}$, but belong to different pairs.
\end{itemize}
\end{lemma}

\begin{proof}
We obtain $(\smash{\widetilde{\Gamma}_{k'}(\pi)},\smash{\widetilde{C}(\pi)})$ from $(\Gamma_k,C(\pi))$ by assigning all edges of $\Gamma_k$ the initial weight zero and carrying out the following reduction algorithm.

\begin{itemize}
\item[Step I] Check the set $C(\pi)$ for pairs that involve adjacent edges. If there are any, go through them in the order they appear in $C(\pi)$, replace the corresponding edges of~$\Gamma_k$ and their common vertex by a single edge and remove the pair from $C(\pi)$. Any new edges created in this step are assigned the weight one. This ensures (a) and (c).
\item[Step II] Check the new graph for any edges of nonzero weight that are adjacent. If there are any, identify the edges that involve the vertex $a_l$ for the smallest value $l\in\{1,\dots,k\}$, replace the two edges and their common vertex by a single edge and assign it the sum of the weights of the edges that were replaced. This step is repeated until (b) holds. The validity of (a) and (c) is not changed in the process.
\item[Step III] Check the new graph for any edges assigned a nonzero weight, say $w_j>0$, that are adjacent to two edges belonging to the same pair. If there are any, go through the corresponding pairs in the order they appear in $C(\pi)$ and replace the pair and the weighted edge, as well as the two vertices between them by a single edge. After going through $C(\pi)$ once, remove the pairs that were replaced from $C(\pi)$ and assign the new edges the respective weights $w_j+1$. As this may generate new subgraphs of the same structure, repeat the step until (d) is satisfied. This does not interfere with~(a) and (c), but~(b) may not hold any more.
\item[Step IV] Repeat Steps II and III until neither can be carried out any further. The graph resulting from this procedure satisfies (b) and (d) simultaneously.
\end{itemize}
Denoting the (weighted) graph resulting from this procedure by $\smash{\widetilde{\Gamma}_{k'}(\pi)}$ and the collection of pairs that remain in $C(\pi)$ after Step IV by $\smash{\widetilde{C}(\pi)}$, we obtain a tuple that satisfies~(a)-(d).
\end{proof}

Note that the above algorithm removes all subgraphs that correspond to a non-crossing partition of a subset of $\{1,\dots,k\}$ and replaces them by an edge with nonzero weight each. In particular, $\smash{\widetilde{C}(\pi)}=\emptyset$ whenever the pairing~$\pi$ is non-crossing. In this case only one vertex connected to itself by an edge (loop) of weight~$k/2$ remains, as the weight assigned to an edge reflects the number of pairings removed from $C(\pi)$ in obtaining it. We demonstrate the algorithm for the example given on the right of Fig.~\ref{fig2} below. For simplicity, the vertex labels and edge weights that are equal to zero are left out.

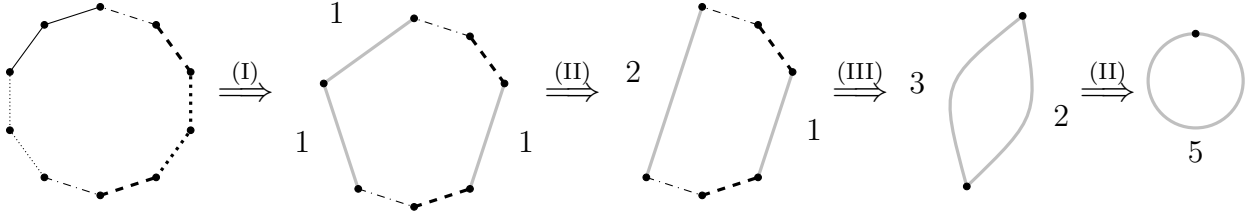
\begin{figure}[H]
\begin{center}
\begin{tikzpicture}[scale=1.25,baseline=(current bounding box.center)]
\draw[black] (0,1) -- (-0.5878,0.809);
\draw[black] (-0.5878,0.809) -- (-0.951,0.309);
\draw[black, densely dotted] (-0.951,0.309) -- (-0.951,-0.309);
\draw[black, densely dotted] (-0.951,-0.309) -- (-0.5878,-0.809);
\draw[black, dashdotted] (-0.5878,-0.809) -- (0,-1);
\draw[black,dashed,very thick] (0,-1) -- (0.5878,-0.809);
\draw[black,dotted,very thick] (0.5878,-0.809) -- (0.951,-0.309);
\draw[black,dotted,very thick] (0.951,-0.309) -- (0.951,0.309);
\draw[black,dashed,very thick] (0.951,0.309) -- (0.5878,0.809);
\draw[black, dashdotted] (0.5878,0.809) -- (0,1);
\filldraw [black] (0,1) circle (1pt)(-0.5878,0.809) circle (1pt)(-0.951,0.309) circle (1pt)(-0.951,-0.309) circle (1pt)(-0.5878,-0.809) circle (1pt);
\filldraw [black] (0,-1) circle (1pt)(0.5878,-0.809) circle (1pt)(0.951,-0.309) circle (1pt)(0.951,0.309) circle (1pt)(0.5878,0.809) circle (1pt);\end{tikzpicture}\hspace{0.3cm}\scalebox{1.1}{$\overset{\mathrm{(I)}}\Longrightarrow$}
\begin{tikzpicture}[scale=1.25,baseline=(current bounding box.center)]
\draw (-0.5878,0.809) node[above left=1pt] {1};
\draw (-0.951,-0.309) node[left=1pt] {1};
\draw (0.951,-0.309) node[right=1pt] {1};
\draw[lightgray,very thick] (0,1) -- (-0.951,0.309);
\draw[lightgray,very thick] (-0.951,0.309) -- (-0.5878,-0.809);
\draw[black,dashdotted] (-0.5878,-0.809) -- (0,-1);
\draw[black,dashed, very thick] (0,-1) -- (0.5878,-0.809);
\draw[lightgray,very thick] (0.5878,-0.809) -- (0.951,0.309);
\draw[black,dashed, very thick] (0.951,0.309) -- (0.5878,0.809);
\draw[black,dashdotted] (0.5878,0.809) -- (0,1);
\filldraw [black] (0,1) circle (1pt)(-0.951,0.309) circle (1pt)(-0.5878,-0.809) circle (1pt);
\filldraw [black] (0,-1) circle (1pt)(0.5878,-0.809) circle (1pt)(0.951,0.309) circle (1pt)(0.5878,0.809) circle (1pt);
\end{tikzpicture}\scalebox{1.1}{$\overset{\mathrm{(II)}}\Longrightarrow$}
\begin{tikzpicture}[scale=1.25,baseline=(current bounding box.center)]
\draw (-0.951,0.309) node[right=1pt] {2};
\draw (0.951,-0.309) node[right=1pt] {1};
\draw[lightgray,very thick] (0,1) -- (-0.5878,-0.809);
\draw[black,dashdotted] (-0.5878,-0.809) -- (0,-1);
\draw[black,dashed, very thick] (0,-1) -- (0.5878,-0.809);
\draw[lightgray,very thick] (0.5878,-0.809) -- (0.951,0.309);
\draw[black,dashed, very thick] (0.951,0.309) -- (0.5878,0.809);
\draw[black,dashdotted] (0.5878,0.809) -- (0,1);
\filldraw [black] (0,1) circle (1pt)(-0.5878,-0.809) circle (1pt);
\filldraw [black] (0,-1) circle (1pt)(0.5878,-0.809) circle (1pt)(0.951,0.309) circle (1pt)(0.5878,0.809) circle (1pt);
\end{tikzpicture}\scalebox{1.1}{$\overset{\mathrm{(III)}}\Longrightarrow$}
\begin{tikzpicture}[scale=1.25,baseline=(current bounding box.center)]
\draw (-0.2939,0.0955) node[left=1pt] {3};
\draw (0.7694,-0.25) node[right=1pt] {2};
\draw[lightgray,very thick] (0,-1) .. controls (-0.2939,0.0955) .. (0.5878,0.809);
\draw[lightgray,very thick] (0.5878,0.809) .. controls (0.7694,-0.25) .. (0,-1);
\filldraw [black] (0,-1) circle (1pt)(0.5878,0.809) circle (1pt);
\end{tikzpicture}\scalebox{1.1}{$\overset{\mathrm{(II)}}\Longrightarrow$}
\begin{tikzpicture}[scale=1.25,baseline=(current bounding box.center)]
\draw (0,0) node[below=1pt] {5};
\draw[lightgray,very thick] (0,0.5) circle (.5cm);
\filldraw [black] (0,1) circle (1pt);
\end{tikzpicture}
\caption{The steps of the reduction algorithm with arrows indicating the steps.}\label{fig3}
\end{center}
\end{figure}

\vspace{-0.7cm}
Observe that there is a one-to-one correspondence between the graph $(\smash{\widetilde{\Gamma}_{k'}(\pi)},\smash{\widetilde{C}(\pi)})$ obtained in Lemma~\ref{rewriting-graph} and a suitable right-hand side of~\eqref{rewrite-goal}. By considering the algebraic counterpart of the graph algorithm, we obtain a bound for the norm of the matrices involved.

\begin{lemma}\label{rewriting-cumu}
Let $k\geq2$ be even. For every pairing $\pi\in\Pi_k$ exists an expression of the form~\eqref{rewrite-goal} with matrices $M^{(i)}$, $i\in J$, that satisfy $\|M^{(i)}\|\leq CN^{w_i}$. The numbers $w_i\geq1$ correspond to the edge weights of the graph $(\smash{\widetilde{\Gamma}_{k'}(\pi)},\smash{\widetilde{C}(\pi)})$ that satisfies (a)-(d) of Lemma~\ref{rewriting-graph}.
\end{lemma}

\begin{proof}
Starting with the left-hand side of~\eqref{rewrite-goal}, we carry out the algorithm from Lemma~\ref{rewriting-graph} on the corresponding graph and note the analogous steps for the 2-cumulants. The matrices~$M^{(i)}$, $i\in J$, are defined inductively along the rewriting procedure. Recall that every pairing in~$C(\pi)$ corresponds to a 2-cumulant and every vertex in $\Gamma_k$ represents a summation over an index $a_1,\dots,a_k$. An iteration of Step I thus corresponds to replacing
\begin{displaymath}
\sum_{a_{l+1}}\kap{2}(a_la_{l+1},a_{l+1}a_{l+2})=T_{a_l,a_{l+2}},
\end{displaymath}
i.e., applying~\eqref{def-T} to carry out a summation over an internal index. Hence, every new edge of weight one in Step I corresponds to introducing one matrix element of $T$ and defining
\begin{displaymath}
M^{(i)}:=T
\end{displaymath}
for all $i\in J$. The vertex that is removed from the graph matches the index over which the summation was carried out. Recalling that $\|T\|\leq CN$, the order of $N$ in the norm bound matches the edge weight as claimed. 

\medskip
Step~II corresponds to matrix multiplication. Let $M^{(i_1)}$ and $M^{(i_2)}$ denote two matrices corresponding to adjacent edges of nonzero weight in the graph and let $a_l,a_m,a_n\in\{a_1,\dots,a_k\}$ with~$a_l\neq a_m$ and~$a_m\neq a_n$ denote the vertices adjacent to them. An iteration of Step II corresponds to replacing
\begin{displaymath}
\sum_{a_m}M^{(i_1)}_{a_l,a_m}M^{(i_2)}_{a_m,a_n}=(M^{(i_1)}M^{(i_2)})_{a_l,a_n}=:M^{(i')}_{a_l,a_n},
\end{displaymath}
where $M^{(i')}$ denotes the matrix corresponding to the new edge and the vertex that is removed matches the summation index. Since~${\|M^{(i')}\|\leq\|M^{(i_1)}\|\ \|M^{(i_2)}\|}$, we obtain a norm bound from the initial estimates. In particular, the exponents of $N$ from the previous bounds are added, matching the edge weight prescribed by the graph algorithm.

\medskip
For Step III, let $M^{(i)}$ denote the matrix corresponding to an edge of weight $w_i>0$ and assume that the two adjacent edges belong to the same pair, which is represented by the~2-cumulant~$\kap{2}(a_la_{l+1},a_ma_{m+1})$. An iteration of Step III corresponds to replacing
\begin{align}\label{new-matrix}
\sum_{a_{l+1},a_m}M^{(i)}_{a_{l+1},a_m}\kap{2}(a_la_{l+1},a_ma_{m+1})=:M^{(i')}_{a_l,a_{m+1}},
\end{align}
where the matrix $M^{(i')}$ corresponds to the edge of weight $w_i+1$ introduced in the graph algorithm and the summation indices $a_{l+1},a_m$ match the vertices that are removed. Note that the norm of the new matrix can be estimated by $\|M^{(i')}\|\leq N\|M^{(i)}\|\leq CN^{w_i+1}$ following an argument similar to~\eqref{norm-B} such that the power of $N$ in the bound again matches the prescribed edge weight. We have thus shown that the reduction algorithm on $(\Gamma_k,C(\pi))$ from Lemma~\ref{rewriting-graph} translates to a rewriting procedure for the corresponding cumulant term. In particular, an edge with weight $w_i$ appearing in any step of the graph algorithm can be associated with a matrix $M^{(i)}$ with $\|M^{(i)}\|\leq CN^{w_i}$.
\end{proof}

\begin{remark}
Lemma~\ref{rewriting-cumu} extends to arbitrary $k\in\N$ and partitions $\pi\in\Pi_k$ by restricting the algorithm in Lemma~\ref{rewriting-graph} to the pairs of $C(\pi)$. The resulting graph~$(\smash{\widetilde{\Gamma}_{k'}(\pi)},\smash{\widetilde{C}(\pi)})$ satisfies~{(a)-(d)} with the properties (c) and (d) restricted to the pairs in~$\smash{\widetilde{C}(\pi)}$. Further,~\eqref{rewrite-goal} generalizes to the form
\begin{equation}\label{rewrite-goal-generalized}
\sum_{a_1,\dots,a_k}\prod_{B\in\pi}\kap{|B|}(a_ja_{j+1}|j\in B)=\sum_{b_1,\dots,b_{k'}}\Big(\prod_{B\in\pi'}\kap{|B|}(b_jb_{j+1}|j\in B)\prod_{i\in J}M^{(i)}_{b_{i_1},b_{i_2}}\Big)
\end{equation}
where none of the 2-cumulants on the right-hand side of~\eqref{rewrite-goal-generalized} involve internal indices and the indices $b_{i_1},b_{i_2}$ never appear in the same 2-cumulant.
\end{remark}

After the rewriting procedure, summations over $k'\leq k$ indices remain. Matching~\eqref{rewrite-goal}, we rename them as~${b_1,\dots, b_{k'}}$ to indicate that the rewriting step is completed.

\medskip
\underline{\smash{Step 2: Estimates via summing-in steps}}\\
Recall that the reduction algorithm in Lemma~\ref{rewriting-graph} terminates in one of two possible states: either $\smash{\widetilde{C}(\pi)}=\emptyset$ (if $\pi$ is non-crossing) or $\smash{\widetilde{C}(\pi)}$ contains at least two pairings (if $\pi$ is crossing). We estimate the corresponding contributions to~\eqref{2-cumu-term} separately.

\begin{lemma}[Leading Terms]\label{estimates1}
Let $k\geq2$ be even and $\pi\in\Pi_k$ be a non-crossing pairing. Then
\begin{displaymath}
N^{-k/2-1}\Big|\sum_{a_1,\dots,a_k}\prod_{B\in\pi}\kap{2}(a_ja_{j+1}|j\in B)\Big|\leq C.
\end{displaymath}
\end{lemma}

\begin{proof}
Whenever $\pi$ is non-crossing, the algorithm in Lemma~\ref{rewriting-graph} reduces $(\Gamma_k,C(\pi))$ to a single vertex $b_1$ connected to itself by an edge of weight $k/2$. By the rewriting part of Lemma~\ref{rewriting-cumu}, we thus have
\begin{displaymath}
\sum_{a_1,\dots,a_k}\prod_{B\in\pi}\kap{2}(a_ja_{j+1}|j\in B)=\sum_{b_1}M_{b_1,b_1}=\tr(M),
\end{displaymath}
where~$M$ denotes the matrix corresponding to the loop in $\smash{\widetilde{\Gamma}_{k'}(\pi)}$. Further,~${\|M\|\leq CN^{k/2}}$, giving $|\tr(M)|\leq CN^{k/2+1}$ and the claim.
\end{proof}

\begin{lemma}[Subleading Pairings]\label{estimates2}
Let $k\geq2$ be even and $\pi\in\Pi_k$ be a crossing pairing. Then
\begin{displaymath}
N^{-k/2-1}\Big|\sum_{a_1,\dots,a_k}\prod_{B\in\pi}\kap{2}(a_ja_{j+1}|j\in B)\Big|\leq CN^{-1/2}.
\end{displaymath}
\end{lemma}

\begin{proof}
Whenever $\pi$ is crossing, at least two pairings remain in $\smash{\widetilde{C}(\pi)}$ after the algorithm in Lemma~\ref{rewriting-graph} terminates. As the resulting term is claimed to be subleading, we aim to show~\eqref{CR} for the entire right-hand side of~\eqref{rewrite-goal}, possibly allowing an additional summation, i.e., another factor of~$\sqrt{N}$. Let $M^{(i)}$ denote a matrix that corresponds to an edge of weight $w_i$ in~$\smash{\widetilde{\Gamma}_{k'}(\pi)}$, i.e., $\|\smash{M^{(i)}}\|\leq CN^{w_i}$ following Lemma~\ref{rewriting-cumu}. By defining
\begin{equation}\label{def-vector-t}
\Big(\mathbf{x}^{[M^{(i)}]b_l}\Big)_{b_{l+1}}:=\frac{1}{N^j}M^{(i)}_{b_l,b_{l+1}},\ b_{l+1}=1,\dots,N
\end{equation}
with $l\in\{1,\dots,k'\}$ and $k'+1=1$, we obtain a vector that satisfies
\begin{equation}\label{uni-bound-vector}
\|\mathbf{x}^{[M^{(i)}]b_l}\|_2\leq C
\end{equation}
uniformly for any choice of~$b_l$. Going through the index set $J$, we rewrite every matrix element on the right-hand side of~\eqref{rewrite-goal} as a normalized vector and collect any separate powers of $N$ in front of the sum.

\medskip
Next, we sum in the vectors using the notation of Lemma~\ref{ext-red-rule-2-cumu}. Going through the newly introduced vectors~\eqref{def-vector-t} in increasing order of the corresponding~$l$, we replace
\begin{displaymath}
\sum_{b_{l+1}}\Big(\mathbf{x}^{[M^{(i)}]b_l}\Big)_{b_{l+1}}\kap{2}(b_{l+1}b_{l+2},b_mb_{m+1})=\kap{2}(\mathbf{x}^{[M^{(i)}]b_l}b_{l+2},b_mb_{m+1}),
\end{displaymath}
where $b_m\neq b_{l-1},b_l,b_{l+1},b_{l+2}$ denotes some other summation index. Since all vectors are defined from the rows of the respective matrices in~\eqref{def-vector-t}, we always involve the first index of an index pair and thus replace at most two indices in each~2-cumulant by a vector. As a consequence of Lemmas~\ref{rewriting-graph} and \ref{rewriting-cumu}, the resulting term neither involves summation over internal indices (cf. Step I) nor an index that occurs in the same~2-cumulant both as a superscript of a vector and as an argument (cf. Step III). In Fig.~\ref{fig4} we visualize the rewriting and summing in procedure for the term on the left of Fig.~\ref{fig2}. The small arrows in Fig.~\ref{fig4} point to the index that will be summed in. Again, edge weights that are equal to zero are left out.
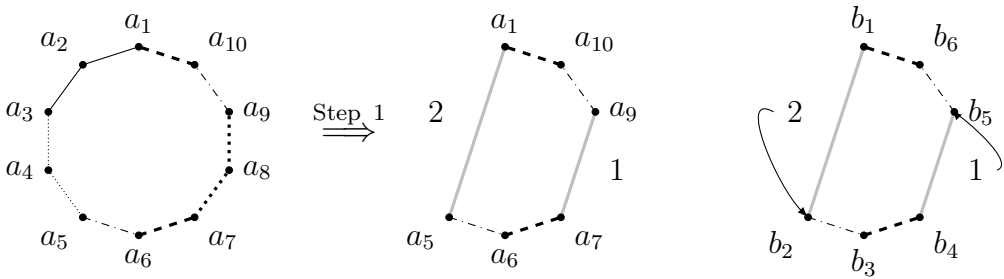
\begin{figure}[H]
\begin{center}
\begin{tikzpicture}[scale=1.25,baseline=(current bounding box.center)] 
\draw (0,1) node[above=1pt] {$a_1$};
\draw (-0.5878,0.809) node[above left=1pt] {$a_2$};
\draw (-0.951,0.309) node[left=1pt] {$a_3$};
\draw (-0.951,-0.309) node[left=1pt] {$a_4$};
\draw (-0.5878,-0.809) node[below left=1pt] {$a_5$};
\draw (0,-1) node[below=1pt] {$a_6$};
\draw (0.5878,-0.809) node[below right=1pt] {$a_7$};
\draw (0.951,-0.309) node[right=1pt] {$a_8$};
\draw (0.951,0.309) node[right=1pt] {$a_9$};
\draw (0.5878,0.809) node[above right=1pt] {$a_{10}$};
\draw[black] (0,1) -- (-0.5878,0.809);
\draw[black] (-0.5878,0.809) -- (-0.951,0.309);
\draw[black, densely dotted] (-0.951,0.309) -- (-0.951,-0.309);
\draw[black, densely dotted] (-0.951,-0.309) -- (-0.5878,-0.809);
\draw[black, dashdotted] (-0.5878,-0.809) -- (0,-1);
\draw[black,dashed,very thick] (0,-1) -- (0.5878,-0.809);
\draw[black,dotted,very thick] (0.5878,-0.809) -- (0.951,-0.309);
\draw[black,dotted,very thick] (0.951,-0.309) -- (0.951,0.309);
\draw[black, dashdotted] (0.951,0.309) -- (0.5878,0.809);
\draw[black,dashed,very thick] (0.5878,0.809) -- (0,1);
\filldraw [black] (0,1) circle (1pt)(-0.5878,0.809) circle (1pt)(-0.951,0.309) circle (1pt)(-0.951,-0.309) circle (1pt)(-0.5878,-0.809) circle (1pt);
\filldraw [black] (0,-1) circle (1pt)(0.5878,-0.809) circle (1pt)(0.951,-0.309) circle (1pt)(0.951,0.309) circle (1pt)(0.5878,0.809) circle (1pt);\end{tikzpicture}\hspace{0.4cm}\scalebox{1.1}{$\overset{\mathrm{Step\ 1}}{\Longrightarrow}$}
\begin{tikzpicture}[scale=1.25,baseline=(current bounding box.center)] 
\draw (0,1) node[above=1pt] {$a_1$};
\draw (-0.5878,-0.809) node[below left=1pt] {$a_5$};
\draw (0,-1) node[below=1pt] {$a_6$};
\draw (0.5878,-0.809) node[below right=1pt] {$a_7$};
\draw (0.951,0.309) node[right=1pt] {$a_9$};
\draw (0.5878,0.809) node[above right=1pt] {$a_{10}$};
\draw (-0.951,0.309) node[right=1pt] {2};
\draw (0.951,-0.309) node[right=1pt] {1};
\draw[lightgray,very thick] (0,1) -- (-0.5878,-0.809);
\draw[black,dashdotted] (-0.5878,-0.809) -- (0,-1);
\draw[black,very thick,dashed] (0,-1) -- (0.5878,-0.809);
\draw[lightgray,very thick] (0.5878,-0.809) -- (0.951,0.309);
\draw[black,dashdotted] (0.951,0.309) -- (0.5878,0.809);
\draw[black,very thick,dashed] (0.5878,0.809) -- (0,1);
\filldraw [black] (0,1) circle (1pt)(-0.5878,-0.809) circle (1pt);
\filldraw [black] (0,-1) circle (1pt)(0.5878,-0.809) circle (1pt)(0.951,0.309) circle (1pt)(0.5878,0.809) circle (1pt);\end{tikzpicture}\hspace{1cm}
\begin{tikzpicture}[scale=1.25,baseline=(current bounding box.center)] 
\draw (0,1) node[above=1pt] {$b_1$};
\draw (-0.5878,-0.809) node[below left=1pt] {$b_2$};
\draw (0,-1) node[below=1pt] {$b_3$};
\draw (0.5878,-0.809) node[below right=1pt] {$b_4$};
\draw (0.951,0.309) node[right=1pt] {$b_5$};
\draw (0.5878,0.809) node[above right=1pt] {$b_6$};
\draw (-0.951,0.309) node[right=1pt] {2};
\draw (0.951,-0.309) node[right=1pt] {1};
\draw[lightgray,very thick] (0,1) -- (-0.5878,-0.809);
\draw[black,dashdotted] (-0.5878,-0.809) -- (0,-1);
\draw[black,very thick,dashed] (0,-1) -- (0.5878,-0.809);
\draw[lightgray,very thick] (0.5878,-0.809) -- (0.951,0.309);
\draw[black,dashdotted] (0.951,0.309) -- (0.5878,0.809);
\draw[black,very thick,dashed] (0.5878,0.809) -- (0,1);
\draw [-{latex}, black] (-0.951,0.309) to [out=150] (-0.5878,-0.809);
\draw [-{latex reversed},black] (1.4,-0.309) to [out=20] (1,0.25);
\filldraw [black] (0,1) circle (1pt)(-0.5878,-0.809) circle (1pt);
\filldraw [black] (0,-1) circle (1pt)(0.5878,-0.809) circle (1pt)(0.951,0.309) circle (1pt)(0.5878,0.809) circle (1pt);
\end{tikzpicture}
\caption{Visualization of the rewriting and summing in for a crossing partition.}\label{fig4}
\end{center}
\end{figure}

\vspace{-0.7cm}
The term obtained after the sum-in procedure is now tractable by recursive summation as illustrated in Examples~\ref{ex-nointindices} and~\ref{ex-intindices}. We give the algorithm below. 
\begin{itemize}
\item[Step i] Identify the index $b_l$ in the sum with the smallest subscript~$l$. After the rewriting and sum-in steps, $b_l$ occurs in exactly two 2-cumulants. We distinguish two cases:
\begin{itemize}
\item Whenever both $b_lb_{l+1}$ and $b_{l-1}b_l$ occur as index pairs in different 2-cumulants, apply an estimate similar to~\eqref{start-sum} with $b_l$ in place of $a_1$, and start the procedure with the~2-cumulant that involves the pair~$b_lb_{l+1}$.
\item Whenever $b_l$ occurs as a superscript of a vector in one 2-cumulant and as an argument in another, start the procedure with the 2-cumulant that involves $b_l$ as an argument.
\end{itemize}
\item[Step ii] We distinguish two cases depending on the choice in Step i:
\begin{itemize}
\item Whenever the 2-cumulant does not involve any vectors, identify its four indices. Separate the summations over these indices from the remainder of the sum by estimating any occurrences of them in other 2-cumulants by a suitable maximum.
\item Whenever the 2-cumulant involves vectors, identify the indices that either occur as the superscript of a vector or as arguments. Treat any summation over indices that appear as arguments similar to the first case, then consider the maximum for the remaining indices to isolate the term from the remainder of the sum~(see~\eqref{eq-ex-withvectors}).
\end{itemize}
Next, perform the summations that have been isolated using the uniform bound~\eqref{uni-bound-vector} for any vectors that appear. Note that all estimates obtained from Lemmas~\ref{red-rule-2-cumu} and~\ref{ext-red-rule-2-cumu} follow~\eqref{CR}, i.e., the final bound in each iteration reduces to a counting argument.
\item[Step iii] Identify the index $b_{l'}$ in the sum which has the smallest subscript~$l'$ and carry out Step~ii for the 2-cumulant that involves $b_{l'}$ as an argument. Repeat the procedure until all summations have been carried out.
\end{itemize}

The above argument shows that all summations remaining after the rewriting and sum-in procedure follow~\eqref{CR}, possibly up to a factor of~$\sqrt{N}$ if a summation is added in Step~i. Hence, the claim follows by showing that~\eqref{CR} stays valid when the remaining summations as well as the factors~$N^{w_i}$ obtained from~\eqref{def-vector-t} are taken into account.
First, observe that~${\|T^j\|\leq CN^j}$, while a matrix element of~$T^j$ is obtained by evaluating
\begin{displaymath}
T^j_{a_l,a_{l+2j}}=\sum_{a_{l+1},\dots,a_{l+2j-1}}\kap{2}(a_la_{l+1},a_{l+1}a_{l+2})\dots\kap{2}(a_{l+2j-2}a_{l+2j-1},a_{l+2j-1}a_{l+2j})
\end{displaymath}
i.e., performing summations over~${2j-1}$ consecutive indices. Hence, these matrix elements yield a factor $\sqrt{N}$ more than prescribed by~\eqref{CR} if estimated trivially. Moreover, carrying out a rewriting step corresponding to Step II or III on the graph does not change this fact. Indeed, if the matrix elements of $M^{(1)}$ and~$M^{(2)}$ are obtained from carrying out~${2j_1-1}$ and~${2j_2-1}$ summations and the matrices are bounded by $CN^{j_1}$ and~$CN^{j_2}$ in norm, respectively, then
\begin{displaymath}
(M^{(1)}M^{(2)})_{a,c}=\sum_{b}M^{(1)}_{a,b}M^{(2)}_{b,c},
\end{displaymath}
accounts for~$2(j_1+j_2)-1$ summations while $\|M^{(1)}M^{(2)}\|\leq CN^{j_1+j_2}$. Similarly, if the matrix elements of~$M^{(i)}$ are obtained from carrying out $2j-1$ summations and $\|M^{(i)}\|\leq CN^j$, then the matrix elements of $M^{(i')}$ in~\eqref{new-matrix} account for two more summations, i.e., $2(j+1)-1$ in total, while the power of~$N$ in the norm bound is increased by one. Since replacing an index by a vector of the form~\eqref{def-vector-t} after the rewriting of the term also requires carrying out one summation, the additional factor of $\sqrt{N}$ per matrix is balanced out.

\medskip
Hence, at most $k+1$ powers of $\sqrt{N}$ are collected from bounding the~$k$ summations with the extra factor~$\sqrt{N}$ being obtained whenever an estimate similar to~\eqref{start-sum} is used in Step~i. Thus, the bound is of order $N^{-1/2}$ here, giving the claim and showing that all crossing pairings are subleading. 
\end{proof}

\vspace{-4mm}
\section{Proof of~\eqref{est-by-const} in the General Case}\label{sect-hi-cumu-estimates}
Let now $k\geq1$ be arbitrary. We aim to show that
\begin{equation}\label{k-cumu-term}
N^{-k/2-1}\Big|\sum_{a_1,\dots,a_k}\prod_{B\in\pi}\kap{|B|}(a_ja_{j+1}|j\in B)\Big|\leq C(k)
\end{equation}
for all $\pi\in\Pi_k$. This would complete the proof of~\eqref{est-by-const} via~\eqref{cumu-expansion}. We start by proving the necessary bounds for $j$-cumulants with~${j\geq3}$ and give the proof of~\eqref{k-cumu-term} in Section~\ref{sect-put-together}. Note that no internal index is summed~up below.
\begin{lemma}\label{red-rule-3-cumu}
Assume that \eqref{A3} holds. Then
\begin{align}
\sum_{a_1,a_2,a_3,a_4}|\kap{3}(a_1a_2,a_3a_4,a_5a_6)|&\leq C,\label{eq-rr3c-1}\\
\sum_{a_1,\dots,a_5}|\kap{3}(a_1a_2,a_3a_4,a_5a_6)|&\leq CN,\label{eq-rr3c-2}\\
\sum_{a_1,\dots,a_6}|\kap{3}(a_1a_2,a_3a_4,a_5a_6)|&\leq CN^2\label{eq-rr3c-3}
\end{align}
uniformly for any choice of the unsummed indices. In particular, the estimates follow~\eqref{CR}.
\end{lemma}
\begin{proof}
Applying~\eqref{A3-k-cumu} to the 3-cumulant in~\eqref{eq-rr3c-1} and taking the maximum in a suitable way in the resulting terms, we obtain, e.g.,
\begin{align*}
&\sum_{a_1,\dots,a_4}|\kap{2}(a_1a_2,a_3a_4)\kap{2}(a_3a_4,a_5a_6)|\leq \Big(\max_{a_3,a_4}\sum_{a_1,a_2}|\kap{2}(a_1a_2,a_3a_4)|\Big)\sum_{a_3,a_4}|\kap{2}(a_3a_4,a_5a_6)|,
\end{align*}
as well as similar bounds for the terms corresponding to the other possible spanning trees. Hence, the summation over up to two index pairs is bounded by a constant, giving~\eqref{eq-rr3c-1}. The remaining estimates~\eqref{eq-rr3c-2} and~\eqref{eq-rr3c-3} readily follow, as the additional summations yield a factor of $N$ or $N^2$, respectively.
\end{proof}
Observe that the bounds given in~\eqref{eq-rr3c-1} and~\eqref{eq-rr3c-2} imply that, e.g., 
\begin{align}
\sum_{a_2}|\kap{3}(a_1a_2,a_2a_3,a_4a_5)|\leq\sum_{a_2,a_2'}|\kap{3}(a_1a_2,a_2'a_3,a_4a_5)|\leq C\leq C\sqrt{N}\label{dummy-ind}
\end{align}
uniformly for any choice of the unsummed indices. In particular, the estimates comply with~\eqref{CR} and there is no need to perform summations over internal indices separately as in Section~\ref{sect-2-cumu-estimates}. The only exception occurs for $k=3$, where one obtains
\begin{equation}\label{k3-term}
\sum_{a_1,a_2,a_3}|\kap{3}(a_1a_2,a_2a_3,a_3a_1)|\leq\sum_{a_1,a_2,a_3,a_1',a_2',a_3'}|\kap{3}(a_1a_2,a_2'a_3,a_3'a_1')|\leq CN^2
\end{equation}
instead of the bound of order $N^{3/2}$ prescribed by the counting rule. As~\eqref{k3-term} is the only term for $k=3$ due to \eqref{A1}, and $k/2+1=5/2$, it follows that~\eqref{est-by-const} holds for $k=3$. Hence, we can exclude the $k=3$  case from the following analysis and assume that all estimates for~3-cumulants with or without internal indices comply with~\eqref{CR}. We further note the following definition.
\begin{definition}
For $\mathbf{x}\in\R^{N}$, set
\begin{equation}\label{3-cumu-vector}
\kap{3}(\mathbf{x}a_2,a_3a_4,a_5a_6):=\sum_{a_1}\kap{3}(a_1a_2,a_3a_4,a_5a_6)x_{a_1}
\end{equation}
and define any 3-cumulants with one or more indices replaced by a vector analogously.
\end{definition}

Again, we follow the convention that vectors only occur as the first index of every index pair. Similarly to Lemma~\ref{ext-red-rule-2-cumu}, we obtain the following estimates.

\begin{lemma}\label{ext-red-rule-3-cumu}
Assume that \eqref{A3} holds and let $\mathbf{x},\mathbf{y},\mathbf{z}\in\R^N$. Then the summation over any number of indices of~$|\kap{3}(\mathbf{x}a_2,a_3a_4,a_5a_6)|$, $|\kap{3}(\mathbf{x}a_2,\mathbf{y}a_4,a_5a_6)|$, or~$|\kap{3}(\mathbf{x}a_2,\mathbf{y}a_4,\mathbf{z}a_6)|$ satisfies~\eqref{CR}. Thus, for one vector we, e.g., have
\begin{align}
\sum_{a_2}|\kap{3}(\mathbf{x}a_2,a_3a_4,a_5a_6)|&\leq CN^{1/2}\|\mathbf{x}\|_2,\label{eq-err3c-1}\\
\sum_{a_2,a_4}|\kap{3}(\mathbf{x}a_2,a_3a_4,a_5a_6)|&\leq CN\|\mathbf{x}\|_2\label{eq-err3c-2}
\end{align}
uniformly for any choice of the unsummed indices. Similar bounds hold if two or three vectors are involved, respectively, e.g.,
\begin{align}
\sum_{a_5}|\kap{3}(\mathbf{x}a_2,\mathbf{y}a_4,a_5a_6)|&\leq CN^{1/2}\|\mathbf{x}\|_2\|\mathbf{y}\|_2,\label{eq-err3c-5}\\
\sum_{a_2}|\kap{3}(\mathbf{x}a_2,\mathbf{y}a_4,\mathbf{z}a_6)|&\leq CN^{1/2}\|\mathbf{x}\|_2\|\mathbf{y}\|_2\|\mathbf{z}\|_2.\nonumber
\end{align}
Moreover, we have bounds that are stronger than~\eqref{CR} by a factor of~$\sqrt{N}$ for every summation that is carried out over two consecutive indices belonging to different index pairs,~e.g.,
\begin{align}
\sum_{a_4,a_5}|\kap{3}(\mathbf{x}a_2,a_3a_4,a_5a_6)|&\leq CN^{1/2}\|\mathbf{x}\|_2.\label{eq-err3c-4}\\
\sum_{a_2,\dots,a_6}|\kap{3}(\mathbf{x}a_2,a_3a_4,a_5a_6)|&\leq CN^{3/2}\|\mathbf{x}\|_2,\label{eq-err3c-3}\\
\sum_{a_4,a_5}|\kap{3}(\mathbf{x}a_2,\mathbf{y}a_4,a_5a_6)|&\leq CN^{1/2}\|\mathbf{x}\|_2\|\mathbf{y}\|_2.\label{eq-err3c-6}
\end{align}
In particular, for 3-cumulants, a summation over an internal index also contributes at most a factor of~$\sqrt{N}$.
\end{lemma}

\begin{proof}
The proof is divided into three general arguments. First, as the bounds obtained in Lemma~\ref{red-rule-3-cumu} are stronger than~\eqref{CR}, estimating the vector elements by $\max_j|v_j|\leq\|\mathbf{v}\|_2$ for~${\mathbf{v}=\mathbf{x},\mathbf{y},\mathbf{z}}$, respectively, and performing the summation directly yields the desired estimate in most cases. Whenever this argument does not yield a sufficiently strong bound, we apply the Cauchy-Schwarz inequality similar to the proof of Lemma~\ref{ext-red-rule-2-cumu}. Lastly, it remains to check that the estimates obtained are strong enough to satisfy the second part of the lemma and also comply with~\eqref{CR} if summation over internal indices is included. As the latter can be treated by introducing additional summation labels similar to~\eqref{dummy-ind}, it is enough to derive suitable bounds for distinct $a_1,\dots,a_6$.

\medskip
Assume first that the 3-cumulant involves only one vector. Here, estimating $|x_j|\leq\|\mathbf{x}\|_2$ and applying~\eqref{eq-rr3c-1} yields
\begin{align}
\sum_{a_2,a_3,a_4}|\kap{3}(\mathbf{x}a_2,a_3a_4,a_5a_6)|&\leq\|\mathbf{x}\|_2\sum_{a_1,\dots,a_4}|\kap{3}(a_1a_2,a_3a_4,a_5a_6)|\leq C\|\mathbf{x}\|_2,\label{eq-err3cp-1}
\end{align}
which immediately implies~\eqref{eq-err3c-1} and~\eqref{eq-err3c-2}. The same $N$-independent bound as~\eqref{eq-err3cp-1} holds if one sums over~$a_2,a_5,a_6$ instead. Whenever all three index pairs are involved in the summation, Lemma~\ref{red-rule-3-cumu} implies a bound of order~$N$,~i.e.,
\begin{align}
\sum_{a_2,\dots,a_5}|\kap{3}(\mathbf{x}a_2,a_3a_4,a_5a_6)|&\leq\|\mathbf{x}\|_2\sum_{a_1,\dots,a_5}|\kap{3}(a_1a_2,a_3a_4,a_5a_6)|\leq CN\|\mathbf{x}\|_2,\label{eq-err3cp-2}
\end{align}
and the same estimate holds if the summation over some other index $a_2,\dots,a_5$ is left out instead of $a_6$. In particular, all summations over two, three and four distinct indices yield a bound of at most order $N$, which complies with~\eqref{CR}. The validity of~\eqref{CR} for summation over five distinct indices follows from~\eqref{eq-rr3c-3}.

\medskip
Now we prove~\eqref{eq-err3c-4} and~\eqref{eq-err3c-3} by applying~\eqref{A3-k-cumu} and estimating one 2-cumulant in the bound similar to~\eqref{1sum-1vector} in the proof of Lemma~\ref{ext-red-rule-2-cumu}. This leads to
\begin{align*}
&\sum_{a_1,a_3,a_5}|x_{a_1}\kap{2}(a_1a_2,a_3a_4)\kap{2}(a_3a_4,a_5a_6)|\\
&\leq\Big(\max_{a_3}\sum_{a_5}|\kap{2}(a_3a_4,a_5a_6)|\Big)\sum_{a_1,a_3}|x_{a_1}\kap{2}(a_1a_2,a_3a_4)|\leq CN^{1/2}\|\mathbf{x}\|_2,
\end{align*}
which, together with similar estimates obtained for the other possible spanning trees on the complete graph on three vertices, proves~\eqref{eq-err3c-4}. The estimate in~\eqref{eq-err3c-3} follows similarly. Applying~\eqref{dummy-ind} and~\eqref{eq-err3cp-2} implies the remaining stronger bounds for the second part of the lemma, which completes the proof for $\kap{3}(\mathbf{x}a_2,a_3a_4,a_5a_6)$.

\medskip
Next, assume that two vectors are involved in~$\kap{3}$. Here, estimating $|x_j|\leq\|\mathbf{x}\|_2$ and~${|y_j|\leq\|\mathbf{y}\|_2}$, applying~\eqref{eq-rr3c-1} yields
\begin{align}
\sum_{a_2,a_4}|\kap{3}(\mathbf{x}a_2,\mathbf{y}a_4,a_5a_6)|&\leq\|\mathbf{x}\|_2\|\mathbf{y}\|_2\sum_{a_1,\dots,a_4}|\kap{3}(a_1a_2,a_3a_4,a_5a_6)|\leq C\|\mathbf{x}\|_2\|\mathbf{y}\|_2.\label{eq-err3cp-3}
\end{align}
Further, arguing similarly to~\eqref{1sum-1vector} gives
\begin{align*}
&\sum_{a_1,a_3,a_5}|x_{a_1}y_{a_3}\kap{2}(a_3a_4,a_5a_6)\kap{2}(a_5a_6,a_1a_2)|\leq CN^{1/2}\|\mathbf{x}\|_2\|\mathbf{y}\|_2,
\end{align*}
which, together with bounds of at most order $N^{1/2}$ for the terms corresponding to the other possible spanning trees, implies~\eqref{eq-err3c-5}. An analogous estimate holds for summation over $a_6$, showing the validity of~\eqref{CR} for a single summation. The remaining cases for summation over distinct indices follow similar to~\eqref{eq-err3cp-3} by applying~\eqref{eq-rr3c-2} and~\eqref{eq-rr3c-3}.

\medskip
Note, however, that the above bounds are again not sufficient to imply~\eqref{CR} for summation over internal indices. Arguing as in~\eqref{1sum-2vectors} yields the stronger estimate
\begin{align*}
&\sum_{a_1,a_3,a_4,a_5}|x_{a_1}y_{a_3}\kap{2}(a_1a_2,a_3a_4)\kap{2}(a_3a_4,a_5a_6)|\leq CN^{1/2}\|\mathbf{x}\|_2\|\mathbf{y}\|_2.
\end{align*}
Again, one can estimate similarly for the terms corresponding to the other possible spanning trees, which implies~\eqref{eq-err3c-6}. Lastly, one needs to check that also
\begin{align*}
\sum_{a_2,a_4,a_5}|\kap{3}(\mathbf{x}a_2,\mathbf{y}a_4,a_4a_5)|&\leq\sum_{a_2,a_4,a_5,a_4'}|\kap{3}(\mathbf{x}a_2,\mathbf{y}a_4,a_4'a_5)|\leq CN^{3/2}\|\mathbf{x}\|_2\|\mathbf{y}\|_2,
\end{align*}
which follows from a similar argument. This completes the proof for~$\kap{3}(\mathbf{x}a_2,\mathbf{y}a_4,a_5a_6)$.

\medskip
For 3-cumulants that involve three vectors, note that writing out the term similar to~\eqref{3-cumu-vector} already involves all three index pairs. Whenever two summations are carried out, the estimates obtained from Lemma~\ref{red-rule-3-cumu} comply with~\eqref{CR},~e.g.,
\begin{align*}
\sum_{a_2,a_4}|\kap{3}(\mathbf{x}a_2,\mathbf{y}a_4,\mathbf{z}a_6)|&\leq\|\mathbf{x}\|_2\|\mathbf{y}\|_2\|\mathbf{z}\|_2\sum_{a_1,\dots,a_5}|\kap{3}(a_1a_2,a_3a_4,a_5a_6)|\leq CN\|\mathbf{x}\|_2\|\mathbf{y}\|_2\|\mathbf{z}\|_2,
\end{align*}
but a different argument is needed for estimating zero, one, or three summations. Here, arguing as in~\eqref{1sum-2vectors} gives
\begin{align*}
&\sum_{a_2}\sum_{a_1,a_3,a_5}|x_{a_1}y_{a_3}z_{a_5}\kap{2}(a_1a_2,a_3a_4)\kap{2}(a_3a_4,a_5a_6)|\\
&\leq\|\mathbf{z}\|_2\Big(\max_{a_3}\sum_{a_5}|\kap{2}(a_3a_4,a_5a_6)|\Big)\sum_{a_1,a_2,a_3}|x_{a_1}y_{a_3}\kap{2}(a_1a_2,a_3a_4)|\leq CN^{1/2}\|\mathbf{x}\|_2\|\mathbf{y}\|_2\|\mathbf{z}\|_2
\end{align*}
together with similar bounds for the terms corresponding to the other possible spanning trees. By a similar argument, we obtain an $N$-independent bound and a bound of order $N$ for zero and three summations, respectively. As internal indices cannot occur here, the proof of the lemma is complete.
\end{proof}
Lastly, we derive the necessary estimates to incorporate cumulants of order four and higher. As no exceptions similar to~\eqref{k3-term} occur, we directly include summations over internal indices into the bound.
\begin{lemma}\label{est-high-cumu}
Assume that \eqref{A3} holds and let $j\geq4$. Then
\begin{align*}
\sum_{a_1,\dots,a_{2j}}|\kap{j}(a_1a_2,\dots,a_{2j-1}a_{2j})|\leq CN^2.
\end{align*}
\end{lemma}

\begin{proof}
After applying~\eqref{A3-k-cumu}, bound the term by considering all possible spanning trees on the complete graph on $j$ vertices. Starting the summation at the leaves of the respective tree, sum up one index pair of each 2-cumulant. Continuing the procedure until the root of the tree is reached, an $N$-independent bound is obtained in every step. Finally, a summation over all four indices of the last 2-cumulant remains, which yields the claimed bound of order $N^2$ by applying the last estimate of Lemma~\ref{red-rule-2-cumu}. 
\end{proof}

Note that Lemma~\ref{est-high-cumu} in particular implies that
\begin{displaymath}
\sum_{a_1,\dots,a_j}|\kap{j}(a_1a_2,a_2a_3,a_3a_4,\dots,a_ja_1)|\leq\sum_{a_1,\dots,a_j,a_1',\dots,a_j'}|\kap{j}(a_1a_2,a_2'a_3,\dots,a_j'a_1')|\leq CN^2
\end{displaymath}
for any $j\geq4$ such that no special treatment is needed for~\eqref{CR}. Moreover, whenever $j\geq5$, the bound is stronger than~\eqref{CR}. However, Lemma~\ref{est-high-cumu} only gives a bound of $N^2$ independent of the number of summations carried out.

\vspace{-2mm}
\subsection{Including 3-cumulants}\label{sect-put-together}
Throughout this section, assume that $\pi\in\Pi_k$ with $|B|\in\{2,3\}$ for all $B\in\pi$. Excluding the cases in Lemma~\ref{2-cumu-lemma}, we obtain the following bound.

\begin{lemma}[Subleading Term with 3-cumulants]\label{2-3-cumu-bound}
Under assumptions \eqref{A1}-\eqref{A3}, let $k\geq1$ and choose a partition $\pi\in\Pi_k$ with~${|B|\in\{2,3\}}$ for all $B\in\pi$ that is not a pairing. Then
\begin{equation}\label{2-3-cumu-term}
N^{-k/2-1}\sum_{a_1,\dots,a_k}\prod_{B\in\pi}|\kap{|B|}(a_ja_{j+1}|j\in B)|\leq CN^{-1/2}.
\end{equation}
\end{lemma}

\begin{proof}[Proof of Lemma~\ref{2-3-cumu-bound}]
We extend the proof Lemma~\ref{estimates2} to include 3-cumulants and show that~\eqref{CR} holds up to a factor of~$\sqrt{N}$. The modifications to each step are given below.

\medskip
\underline{\smash{Step 1: Rewriting}}\\
Since the summation over internal indices in 3-cumulants does not interfere with~\eqref{CR}, we restrict to rewriting the 2-cumulants that involve internal indices and obtain the desired form via~\eqref{rewrite-goal-generalized}. Recall that certain properties only hold for the pairs in~$\smash{\widetilde{C}(\pi)}$ such that, e.g., two edges adjacent to an edge with nonzero weight may still belong to the same~3-tuple. As before, we rename the remaining~$k'$ summation indices as~$b_1,\dots,b_{k'}$ and denote the matrices obtained along the rewriting procedure as $M^{(i)}$ with $i$ in some index set $J$. Recall that $\|M^{(i)}\|\leq CN^{w_i}$ where $w_i\neq0$ is the weight of the corresponding edge in the graph.

\medskip
\underline{\smash{Step 2: Sum-In}}\\
In this step, we focus on the terms specific to 3-cumulants that are not covered by the argument in Lemma~\ref{estimates2}. Consider the possible subgraphs in which an edge with nonzero weight is adjacent to two edges that belong to the same~3-tuple. Here, rewriting and estimating the term using the vectors in~\eqref{def-vector-t} fails similarly to the second part of Example~\ref{ex-intindices}. We visualize the respective subgraphs in Fig.~\ref{fig5} below, where edges belonging to the same~3-tuple are indicated by the same linestyle and~$w_j$~(resp.~$w_{j_1},w_{j_2}$) denotes some nonzero edge weight. For simplicity, edge weights that are equal to zero as well as the labels of most vertices are left out and the remainder of the graph is indicated by horizontal~dots.

\begin{figure}[H]
\begin{center}
$\underbrace{\begin{tikzpicture}[scale=1.25] 
\draw (-0.9239,0.9239) node[left=-7pt] {$w_j$};
\draw (0.3827,0.9239) node[above=1pt] {$b_l$};
\draw (-0.9239,-0.9239) node[left=-7pt] {\color{white}a\color{black}};
\draw (-0.3827,-0.9239) node[below=1pt] {\color{white}b\color{black}};
\draw (0.3827,0) node {$\dots$};
\draw[lightgray,very thick] (-0.3827,0.9239) -- (-0.9239,0.3827);
\draw[black] (-0.3827,0.9239) -- (0.3827,0.9239);
\draw[black] (-0.9239,0.3827) -- (-0.9239,-0.3827);
\draw[black] (-0.9239,-0.3827) -- (-0.3827,-0.9239);
\filldraw [black] (-0.3827,0.9239) circle (1pt)(-0.9239,0.3827) circle (1pt)(-0.9239,-0.3827) circle (1pt)(-0.3827,-0.9239) circle (1pt);
\filldraw [black] (0.3827,0.9239) circle (1pt);
\end{tikzpicture}\hspace{0.5cm}
\begin{tikzpicture}[scale=1.25] 
\draw (-0.9239,0) node[left=1pt] {$w_j$};
\draw (0.3827,0.9239) node[above=1pt] {$b_l$};
\draw (-0.9239,-0.9239) node[left=-7pt] {\color{white}a\color{black}};
\draw (-0.3827,-0.9239) node[below=1pt] {\color{white}b\color{black}};
\draw (0.3827,0) node {$\dots$};
\draw[black] (-0.3827,0.9239) -- (-0.9239,0.3827);
\draw[black] (-0.3827,0.9239) -- (0.3827,0.9239);
\draw[lightgray,very thick] (-0.9239,0.3827) -- (-0.9239,-0.3827);
\draw[black] (-0.9239,-0.3827) -- (-0.3827,-0.9239);
\filldraw [black] (-0.3827,0.9239) circle (1pt)(-0.9239,0.3827) circle (1pt)(-0.9239,-0.3827) circle (1pt)(-0.3827,-0.9239) circle (1pt);
\filldraw [black] (0.3827,0.9239) circle (1pt);
\end{tikzpicture}}_{\text{Type\ 1}}$\hspace{0.5cm}
$\underbrace{\begin{tikzpicture}[scale=1.25] 
\draw (-0.9239,-0.9239) node[left=-7pt] {$w_j$};
\draw (0.3827,0.9239) node[above=1pt] {$b_l$};
\draw (-0.9239,0.3827) node[left=1pt] {$b_m$};
\draw (-0.3827,-0.9239) node[below=1pt] {\color{white}b\color{black}};
\draw (-0.6,0.7) node[left=-5pt] {$\dots$};
\draw (0.3827,0) node {$\dots$};
\draw[lightgray,very thick] (-0.3827,-0.9239) -- (-0.9239,-0.3827);
\draw[black,very thick, dashed] (-0.9239,0.3827) -- (-0.9239,-0.3827);
\draw[black,very thick, dashed] (-0.3827,-0.9239) -- (0.3827,-0.9239);
\draw[black,very thick, dashed] (0.3827,0.9239) -- (-0.3827,0.9239);
\filldraw [black] (-0.3827,0.9239) circle (1pt)(-0.9239,0.3827) circle (1pt)(-0.9239,-0.3827) circle (1pt)(-0.3827,-0.9239) circle (1pt);
\filldraw [black] (0.3827,-0.9239) circle (1pt)(0.3827,0.9239) circle (1pt);
\end{tikzpicture}\hspace{0.5cm}
\begin{tikzpicture}[scale=1.25] 
\draw (-0.9239,0.9239) node[left=-7pt] {$w_j$};
\draw (0.3827,0.9239) node[above=1pt] {$b_l$};
\draw (-0.3827,-0.9239) node[below=1pt] {$b_m$};
\draw (-0.9239,-0.9239) node[left=-7pt] {\color{white}a\color{black}};
\draw (-0.6,-0.7) node[left=-5pt] {$\dots$};
\draw (0.3827,0) node {$\dots$};
\draw[lightgray,very thick] (-0.3827,0.9239) -- (-0.9239,0.3827);
\draw[black,very thick, dashed] (-0.9239,0.3827) -- (-0.9239,-0.3827);
\draw[black,very thick, dashed] (-0.3827,-0.9239) -- (0.3827,-0.9239);
\draw[black,very thick, dashed] (0.3827,0.9239) -- (-0.3827,0.9239);
\filldraw [black] (-0.3827,0.9239) circle (1pt)(-0.9239,0.3827) circle (1pt)(-0.9239,-0.3827) circle (1pt)(-0.3827,-0.9239) circle (1pt);
\filldraw [black] (0.3827,-0.9239) circle (1pt)(0.3827,0.9239) circle (1pt);
\end{tikzpicture}}_{\text{Type\ 2}}$\hspace{0.5cm}
$\underbrace{\begin{tikzpicture}[scale=1.25] 
\draw (-0.9239,0.9239) node[left=-7pt] {$w_{j_1}$};
\draw (-0.9239,-0.9239) node[left=-7pt] {$w_{j_2}$};
\draw (0.3827,0.9239) node[above=1pt] {$b_l$};
\draw (-0.3827,-0.9239) node[below=1pt] {\color{white}b\color{black}};
\draw (0.3827,0) node {$\dots$};
\draw[lightgray,very thick] (-0.3827,0.9239) -- (-0.9239,0.3827);
\draw[black,very thick, dotted] (-0.9239,0.3827) -- (-0.9239,-0.3827);
\draw[lightgray,very thick] (-0.9239,-0.3827) -- (-0.3827,-0.9239);
\draw[black,very thick, dotted] (-0.3827,-0.9239) -- (0.3827,-0.9239);
\draw[black,very thick, dotted] (0.3827,0.9239) -- (-0.3827,0.9239);
\filldraw [black] (-0.3827,0.9239) circle (1pt)(-0.9239,0.3827) circle (1pt)(-0.9239,-0.3827) circle (1pt)(-0.3827,-0.9239) circle (1pt);
\filldraw [black] (0.3827,-0.9239) circle (1pt)(0.3827,0.9239) circle (1pt);
\end{tikzpicture}}_{\text{Type\ 3}}$
\caption{The only subgraphs not covered by previous estimates.}\label{fig5}
\end{center}
\end{figure}

\vspace{-0.7cm}
We gather the 3-cumulant and the matrix element(s) corresponding to the respective subgraphs in Fig.~\ref{fig5} into one term that we refer to as type~1,~2 or~3, respectively. For example, the left type 1 graph translates to the product
\begin{displaymath}
M^{(j)}_{b_{l+1},b_{l+2}}\kap{3}(b_lb_{l+1},b_{l+2}b_{l+3},b_{l+3}b_{l+4})
\end{displaymath}
in the cumulant calculation while the left type 2 graph translates to
\begin{displaymath}
M^{(j)}_{b_{m+1},b_{m+2}}\kap{3}(b_lb_{l+1},b_{m}b_{m+1},b_{m+2}b_{m+3}).
\end{displaymath}
Here, $M^{(j)}$ denotes the matrix that corresponds to the edge with weight $w_j$, respectively. Note that the two solid and two dashed graphs in Fig.~\ref{fig5} look very similar, but are still not identical since the cyclic ordering breaks the symmetry. As a result of Lemma~\ref{rewriting-cumu}, any matrix element that does not occur in a term of type~1,~2 or~3 can be treated using the same sum-in procedure as in the proof of Lemma~\ref{estimates2}, i.e., by applying~\eqref{def-vector-t} and performing the summation over the corresponding $b_{l+1}$ explicitly. Again, defining the vectors by the rows of the respective matrices ensures that only the first index of every index pair in the remaining 2- and~3-cumulants may be replaced by a vector. In particular, estimates for the special cases in Fig.~\ref{fig5} and the counting rule are sufficient to cover any term that is encountered. As an example, consider two interlaced type 1 subgraphs: If the second type 1 term is added between the two solid edges of the first, the result can be estimated using only bounds for type 1 and 2, respectively. Otherwise, the sum-in procedure applies, as at least one edge of nonzero weight is no longer adjacent to two edges from the same grouping.

\medskip
At this point, any separate matrix element occurs in a term of type 1, 2 or~3. Note that also up to one~(types~1 and 3) or two (type 2) indices of the~3-cumulant in the respective terms may have been replaced by a vector along the sum-in procedure. 

\medskip
\underline{\smash{Step 3: Estimates}}\\
To estimate the remaining summations, we apply a recursive summation procedure similar to the proof of Lemma~\ref{estimates2}. Note that the summation over any number of indices of a 2- or~3-cumulant follows~\eqref{CR} by Lemmas~\ref{red-rule-2-cumu},~\ref{ext-red-rule-2-cumu},~\ref{red-rule-3-cumu}, and~\ref{ext-red-rule-3-cumu}, including summations over internal indices in the 3-cumulants and any summation over internal indices of 2-cumulants resulting in the sum-in of the vectors in~\eqref{def-vector-t}. The only terms not covered by the previous arguments are the terms of type 1,~2, and~3. We claim that they, too, follow~\eqref{CR} and show that applying the bounds from Lemma~\ref{red-rule-3-cumu} or~\ref{ext-red-rule-3-cumu} for the 3-cumulant compensates the additional factor(s) of~$\sqrt{N}$ obtained from estimating the matrix element(s) trivially by the norm of the respective matrix.

\medskip
Assume first that the summation on the left-hand side of~\eqref{k-cumu-term} reduces to one single term of type~1, with the internal index occurring, e.g., at $b_4$, or of type~3. We obtain from Lemma~\ref{red-rule-3-cumu} and~\eqref{dummy-ind} that
\begin{align*}
\sum_{b_1,\dots,b_5}|\kap{3}(b_1b_2,b_3b_4,b_4b_5)|\leq CN^2,\quad \sum_{b_1,\dots,b_6}|\kap{3}(b_1b_2,b_3b_4,b_5b_6)|\leq CN^2.
\end{align*}
Observe that these estimates are better than the bound prescribed by~\eqref{CR} by one or two factors of $\sqrt{N}$, respectively. Hence, the final bound for the term follows~\eqref{CR} again.

\medskip
Whenever a term of type~1,~2, or~3 is encountered along the recursive summation procedure, we argue similarly. Recall that Step ii used in the proof of Lemma~\ref{2-cumu-lemma} requires summing over all remaining indices for every 2-cumulant that is encountered. For terms of type~1,~2, or~3, we modify this step and sum up all but one (type~1 and~3) or two (type~2) indices instead, always leaving the index $b_l$ for the largest value of $l$ in each connected subgraph unsummed. The index or indices that are left out always belong to another 2- or~3-cumulant or term of type~1,~2, or~3 and are thus considered in later steps of the procedure. The summation rule is visualized in Fig.~\ref{fig6} below, where the large dots denote the indices that will be summed up after estimating the matrix element(s) trivially.

\begin{figure}[H]
\begin{center}
$\underbrace{\begin{tikzpicture}[scale=1.25] 
\draw (-0.9239,0.9239) node[left=-7pt] {$w_j$};
\draw (-0.9239,-0.9239) node[left=-7pt] {\color{white}a\color{black}};
\draw (0.3827,0) node {$\dots$};
\draw[lightgray,very thick] (-0.3827,0.9239) -- (-0.9239,0.3827);
\draw[black] (-0.3827,0.9239) -- (0.3827,0.9239);
\draw[black] (-0.9239,0.3827) -- (-0.9239,-0.3827);
\draw[black] (-0.9239,-0.3827) -- (-0.3827,-0.9239);
\filldraw [black] (-0.3827,0.9239) circle (2pt)(-0.9239,0.3827) circle (2pt)(-0.9239,-0.3827) circle (2pt)(-0.3827,-0.9239) circle (1pt);
\filldraw [black] (0.3827,0.9239) circle (2pt);
\end{tikzpicture}\hspace{0.5cm}
\begin{tikzpicture}[scale=1.25] 
\draw (-0.9239,0) node[left=1pt] {$w_j$};
\draw (-0.9239,-0.9239) node[left=-7pt] {\color{white}a\color{black}};
\draw (0.3827,0) node {$\dots$};
\draw[black] (-0.3827,0.9239) -- (-0.9239,0.3827);
\draw[black] (-0.3827,0.9239) -- (0.3827,0.9239);
\draw[lightgray,very thick] (-0.9239,0.3827) -- (-0.9239,-0.3827);
\draw[black] (-0.9239,-0.3827) -- (-0.3827,-0.9239);
\filldraw [black] (-0.3827,0.9239) circle (2pt)(-0.9239,0.3827) circle (2pt)(-0.9239,-0.3827) circle (2pt)(-0.3827,-0.9239) circle (1pt);
\filldraw [black] (0.3827,0.9239) circle (2pt);
\end{tikzpicture}}_{\text{Type\ 1}}$\hspace{0.5cm}
$\underbrace{\begin{tikzpicture}[scale=1.25] 
\draw (-0.9239,-0.9239) node[left=-7pt] {$w_j$};
\draw (-0.6,0.7) node[left=-5pt] {$\dots$};
\draw (0.3827,0) node {$\dots$};
\draw[lightgray,very thick] (-0.3827,-0.9239) -- (-0.9239,-0.3827);
\draw[black,very thick,dashed] (-0.9239,0.3827) -- (-0.9239,-0.3827);
\draw[black,very thick,dashed] (-0.3827,-0.9239) -- (0.3827,-0.9239);
\draw[black,very thick,dashed] (0.3827,0.9239) -- (-0.3827,0.9239);
\filldraw [black] (-0.3827,0.9239) circle (1pt)(-0.9239,0.3827) circle (2pt)(-0.9239,-0.3827) circle (2pt)(-0.3827,-0.9239) circle (2pt);
\filldraw [black] (0.3827,-0.9239) circle (1pt)(0.3827,0.9239) circle (2pt);
\end{tikzpicture}\hspace{0.5cm}
\begin{tikzpicture}[scale=1.25] 
\draw (-0.9239,0.9239) node[left=-7pt] {$w_j$};
\draw (-0.9239,-0.9239) node[left=-7pt] {\color{white}a\color{black}};
\draw (-0.6,-0.7) node[left=-5pt] {$\dots$};
\draw (0.3827,0) node {$\dots$};
\draw[lightgray,very thick] (-0.3827,0.9239) -- (-0.9239,0.3827);
\draw[black,very thick,dashed] (-0.9239,0.3827) -- (-0.9239,-0.3827);
\draw[black,very thick,dashed] (-0.3827,-0.9239) -- (0.3827,-0.9239);
\draw[black,very thick,dashed] (0.3827,0.9239) -- (-0.3827,0.9239);
\filldraw [black] (-0.3827,0.9239) circle (2pt)(-0.9239,0.3827) circle (2pt)(-0.9239,-0.3827) circle (1pt)(-0.3827,-0.9239) circle (2pt);
\filldraw [black] (0.3827,-0.9239) circle (1pt)(0.3827,0.9239) circle (2pt);
\end{tikzpicture}}_{\text{Type\ 2}}$\hspace{0.5cm}
$\underbrace{\begin{tikzpicture}[scale=1.25] 
\draw (-0.9239,0.9239) node[left=-7pt] {$w_{j_1}$};
\draw (-0.9239,-0.9239) node[left=-7pt] {$w_{j_2}$};
\draw (0.3827,0) node {$\dots$};
\draw[lightgray,very thick] (-0.3827,0.9239) -- (-0.9239,0.3827);
\draw[black,very thick,dotted] (-0.9239,0.3827) -- (-0.9239,-0.3827);
\draw[lightgray,very thick] (-0.9239,-0.3827) -- (-0.3827,-0.9239);
\draw[black,very thick,dotted] (-0.3827,-0.9239) -- (0.3827,-0.9239);
\draw[black,very thick,dotted] (0.3827,0.9239) -- (-0.3827,0.9239);
\filldraw [black] (-0.3827,0.9239) circle (2pt)(-0.9239,0.3827) circle (2pt)(-0.9239,-0.3827) circle (2pt)(-0.3827,-0.9239) circle (2pt);
\filldraw [black] (0.3827,-0.9239) circle (1pt)(0.3827,0.9239) circle (2pt);
\end{tikzpicture}}_{\text{Type\ 3}}$
\caption{The summation rule.}\label{fig6}
\end{center}
\end{figure}
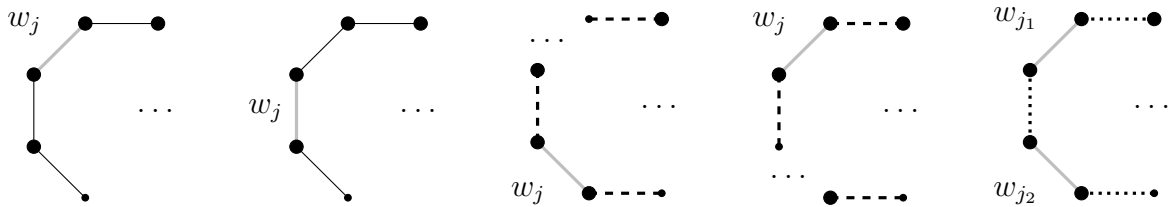

\vspace{-0.7cm}
For the terms of type 1 or~3, leaving one index of the respective~3-cumulant unsummed always allows for a bound of order $N$ by Lemma~\ref{red-rule-3-cumu} or~\eqref{eq-err3cp-2}, i.e., at least $(\sqrt{N})^2$ better than~\eqref{CR} would give, which compensates for the additional factor of~$\sqrt{N}$ or~$N$, respectively, from estimating the matrix elements. The same holds for terms of type 2. Note, however, that due to the dashed subgraphs in Fig.~\ref{fig6} having two connected components, we may also encounter the case that only two summations remain. Here, we apply~\eqref{eq-rr3c-1},~\eqref{eq-err3c-4} or~\eqref{eq-err3c-6} whenever the 3-cumulant involves zero, one or two vectors, respectively. This allows to estimate the remaining two summations by a bound of at most order $N^{1/2}$, showing that~\eqref{CR} holds again. Hence, we can apply a similar recursive summation procedure as used in the proof of Lemma~\ref{2-cumu-lemma} given by the following modified algorithm.

\begin{itemize}
\item[Step i] Identify the terms of type 1, 2 or 3 in the summation. We distinguish two cases.
\begin{itemize}
\item  Whenever terms of type 1, 2 or 3 are present, identify the summation index $b_l$ in them that corresponds to the smallest subscript $l$ .
\item Whenever no terms of type 1, 2 or 3 are present, carry out Step i as in the proof of Lemma~\ref{estimates2}, possibly choosing an index occuring in a~3-cumulant.
\end{itemize}
\item[Step ii] We distinguish two cases depending on the outcome of Step i.
\begin{itemize}
\item Whenever the index $b_l$ occurs in a term of type~1,~2 or~3, estimate the corresponding matrix element(s) using the bound for the matrix norm, then isolate the corresponding~3-cumulant and perform the summation according to Fig.~\ref{fig6}.
\item Whenever the index $b_l$ occurs in a 2- or 3-cumulant, carry out Step~ii as in the proof of Lemma~\ref{estimates2}, possibly applying Lemma~\ref{red-rule-3-cumu} or~\ref{ext-red-rule-3-cumu} if $b_l$ occurs in a 3-cumulant.
\end{itemize}
\item[Step iii] Identify the index $b_{l'}$ in the remaining sum that corresponds to the smallest subscript~$l'$ and carry out Step ii until all summations have been evaluated.
\end{itemize}

Note that starting the summation with estimating a matrix element breaks the cyclic structure of the graph similarly to~\eqref{start-sum}. Further, we have shown that all estimates obtained in Step ii follow~\eqref{CR}. The final bound for the sum obtained at the end of the procedure is thus at most of order $N^{k/2+1/2}$, which is the claim.
\end{proof}

\vspace{-4mm}
\subsection{Including Cumulants of Order Four and Higher}
With the necessary tools established, we proceed to estimating the summation~\eqref{k-cumu-term} in the general case, i.e., when we have cumulants (tuples) of arbitrary order.

\begin{proof}[Proof of~\eqref{k-cumu-term}]
Let $k\in\N$ and $\pi\in\Pi_k$ be arbitrary. Excluding the terms already considered, we can assume that $k\geq4$ and that the partition includes at least one set of four or more elements. As before, we start by identifying any 2-cumulants with internal indices on the right-hand side of~\eqref{k-cumu-term} and rewriting the term using~\eqref{rewrite-goal-generalized}. Next, we carry out the sum-in procedure from Lemma~\ref{2-3-cumu-bound} almost unchanged, i.e.,~\eqref{def-vector-t} is applied for every matrix element that can be incorporated into a 2- or 3-cumulant and the matrix elements occurring in a term of type~1,~2 or 3 (see Fig.~\ref{fig5}) are gathered together with the corresponding 3-cumulant. Note that this step might leave some matrix elements unassigned. We assume first that~$\pi$ is chosen such that this is not the case.

\medskip
Recall that the counting rule established for 2- and 3-cumulants prescribes that every individual summation yields a factor of~$\sqrt{N}$, while the estimate for higher-order cumulants yields a contribution of $N^2$ when summing over all indices in the respective cumulant. To apply both rules simultaneously, we divide the terms obtained from the rewriting procedure into two factors $F$ and $G$, where $G$ collects the 2- and 3-cumulants, as well as the terms of type 1, 2 and 3, and the additional factors of $N$ obtained from the $N^{-j}$ normalization in~\eqref{def-vector-t}, while the higher-order cumulants constitute $F$. Whenever no rewriting is required, we can split the left-hand side of~\eqref{k-cumu-term} directly into
\begin{align*}
F(b_1,\dots,b_k)&:=\prod_{\substack{B\in\pi\\|B|\geq4}}|\kap{|B|}(b_jb_{j+1}|j\in B)|,\quad G(b_1,\dots,b_k):=\prod_{\substack{B\in\pi\\2\leq|B|\leq3}}|\kap{|B|}(b_jb_{j+1}|j\in B)|.
\end{align*}
Next, consider the set $X:=\{j: b_j\text{ appears in a cumulant that belongs to }F\}$. Abbreviating~$b_X=\{b_j: j\in X\}$ and $b_{X^c}=\{b_1,\dots,b_{k'}\}\setminus b_X$, it follows that
\begin{align*}
\sum_{b_1,\dots,b_{k'}}F(b_X)G(b_1,\dots,b_{k'})&\leq\Big(\sum_{b_X}F(b_X)\Big)\Big(\max_{b_X}\sum_{b_{X^c}}G(b_1,\dots,b_{k'})\Big)
\end{align*}
by taking the maximum over all indices in $b_X$ in the factors that occur in $G$.

\medskip
Let $n\geq1$ be the number of factors in~$F$. Since every cumulant included in $F$ is of order four or higher, each factor involves at least eight indices as arguments. As every index belongs to two edges and, hence, has to appear exactly twice in the product of~$F$ and $G$, there are at least $4n$ distinct indices in~$F$. By definition, the total number of distinct indices occurring in~$F$ is equal to~$|X|$ such that $|X|\geq 4n$. Introducing additional summation labels $b_1',\dots,b_l'$ to sum over all indices involved in the respective factors and applying Lemma~\ref{est-high-cumu} thus implies
\begin{align}\label{bound-part1}
\sum_{b_X}F(b_X)\leq\sum_{b_X,b_1',\dots,b_l'}F(b_X,b_1',\dots,b_l')\leq C(N^2)^n\leq CN^{|X|/2}.
\end{align}
For the summation involving $G$, we follow Step 3 from the proof of Lemma~\ref{2-3-cumu-bound}. This yields
\begin{align}\label{bound-part2}
\max_{b_X}\sum_{b_X^c}G(b_1,\dots,b_{k'})\leq CN^{(k-|X|)/2+1/2}
\end{align}
from applying the recursive summation procedure, as all terms in the sum follow~\eqref{CR}. Hence, we obtain a contribution of order~$\sqrt{N}$ for every summation, including the ones that were carried out along the rewriting and sum-in step, and possibly an additional factor $\sqrt{N}$ from an estimate similar to~\eqref{start-sum}. Combining~\eqref{bound-part1} and~\eqref{bound-part2} thus yields
\begin{align*}
N^{-k/2-1}\Big(\sum_{b_X}F(b_X)\Big)\Big(\max_{b_X}\sum_{b_X^c}G(b_1,\dots,b_{k'})\Big)\leq CN^{-1/2}.
\end{align*}

\medskip
In the previously excluded cases, there either is a term similar to the terms of type~1,~2 or~3 with the 3-cumulant replaced by a cumulant of order four or higher, or a matrix element for which introducing vectors as in~\eqref{def-vector-t} would require summing it into a higher-order cumulant. In both cases, we estimate the matrix element trivially by the norm of the corresponding matrix, which yields a factor of~$\sqrt{N}$ more than prescribed by~\eqref{CR}. However, the occurrence of the matrix element implies that the $j$-cumulant must involve at least $j+1$ distinct indices. Hence, one more summation than in the minimal case is carried out when the summation over all indices in the cumulant is evaluated, i.e., the additional factor~$\sqrt{N}$ is balanced out. This implies that the estimates again follow~\eqref{CR}, giving the claim in the general case. In particular, we obtain a sub-leading contribution to~\eqref{cumu-expansion} whenever the term on the left-hand side of~\eqref{k-cumu-term} involves cumulants of order four or higher.
\end{proof}

\vspace{-4mm}
\renewcommand*{\bibname}{References}
\addcontentsline{toc}{chapter}{References}
\bibliographystyle{plain}
\bibliography{References}

\begin{thebibliography}{10}

\bibitem{AdhikariChe2019}
A.~Adhikari and Z.~Che.
\newblock Edge universality of correlated {G}aussians.
\newblock {\em Electron. J. Probab.}, 24, 2019.

\bibitem{AEKS2020}
J.~Alt, L.~Erd\H{o}s, T.~Krüger, and D.~Schröder.
\newblock Correlated random matrices: Band rigidity and edge universality.
\newblock {\em Ann. Probab.}, 48(2):963--1001, 2020.

\bibitem{AndersonGuionnetZeitouni2009}
G.~W. Anderson, A.~Guionnet, and O.~Zeitouni.
\newblock {\em An Introduction to Random Matrices}.
\newblock Cambridge University Press, 2009.

\bibitem{AndersonZeitouni2008}
G.~W. Anderson and O.~Zeitouni.
\newblock A law of large numbers for finite-range dependent random matrices.
\newblock {\em Commun. Pure Appl. Math.}, 61(8):1118–1154, 2008.

\bibitem{BaiYin1988}
Z.~D. Bai and Y.~Q. Yin.
\newblock Necessary and sufficient conditions for almost sure convergence of
  the largest eigenvalue of a {W}igner matrix.
\newblock {\em Ann. Prob.}, 16(4):1729--1749, 1988.

\bibitem{BannaMerlevedePeligrad2015}
M.~Banna, F.~Merlevède, and M.~Peligrad.
\newblock On the limiting spectral distribution for a large class of symmetric
  random matrices with correlated entries.
\newblock {\em Stochastic Process. Appl.}, 125(7):2700–2726, 2015.

\bibitem{BoutetdeMonvelKhorunzhyVasilchuk1996}
A.~Boutet de~Monvel, A.~Khorunzhy, and V.~Vasilchuk.
\newblock Limiting eigenvalue distribution of random matrices with correlated
  entries.
\newblock {\em Markov Process. Related Fields}, 2(4):607–636, 1996.

\bibitem{DuneauIagolnitzerSouillard1973}
M.~Duneau, D.~Iagolnitzer, and B.~Souillard.
\newblock Decrease properties of truncated correlation functions and
  analyticity properties for classical lattices and continuous systems.
\newblock {\em Comm. Math. Phys.}, 31:191–208, 1973.

\bibitem{ErdosMuehlbacher2019}
L.~Erd\H{o}s and P.~M\"{u}hlbacher.
\newblock Bounds on the norm of {W}igner-type random matrices.
\newblock {\em Random Matrices: Theory and Applications}, 8(3):1950009, 2019.

\bibitem{ErdosKruegerSchroeder2017}
L.~Erdős, T.~Krüger, and D.~Schröder.
\newblock Random matrices with slow correlation decay.
\newblock {\em Forum Math. Sigma}, 7, E8, 2019.

\bibitem{RashidiFarOrabyBrycSpeicher2008}
R.~Rashidi Far, T.~Oraby, W.~Bryc, and R.~Speicher.
\newblock On slow-fading {MIMO} systems with nonseparable correlation.
\newblock {\em IEEE Trans. Inform. Theory}, 54(2):544–553, 2008.

\bibitem{FuerediKomlos1981}
Z.~Füredi and J.~Komlós.
\newblock The eigenvalues of random symmetric matrices.
\newblock {\em Combinatorica}, 1:233–241, 1981.

\bibitem{HachemLoubatonNajim2005}
W.~Hachem, P.~Loubaton, and J.~Najim.
\newblock The empirical eigenvalue distribution of a gram matrix: from
  independence to stationarity.
\newblock {\em Markov Process. Related Fields}, 11(4):629–648, 2005.

\bibitem{Juhasz1981}
F.~Juhász.
\newblock On the spectrum of a random graph.
\newblock In {\em Algebraic methods in graph theory. Coll. Math. Soc. J.
  Bolyai}, volume~25, pages 313--316. North Holland, 1981.

\bibitem{MingoSpeicher2017}
J.~A. Mingo and R.~Speicher.
\newblock {\em Free Probability and Random Matrices}.
\newblock Vol. 35, Fields Institute Research Monographs, Springer, New York,
  2017.

\bibitem{Ottolini2017}
M.~Ottolini.
\newblock Spectral norm of random matrices with non-identically distributed
  entries.
\newblock Master's thesis, University Pisa, 2017.

\bibitem{SchenkerSchulzBaldes2005}
J.~H. Schenker and H.~Schulz-Baldes.
\newblock Semicircle law and freeness for random matrices with symmetries or
  correlations.
\newblock {\em Math. Res. Lett.}, 12(4):531–542, 2005.

\bibitem{Vu2007}
V.~H. Vu.
\newblock Spectral norm of random matrices.
\newblock {\em Combinatorica}, 27(6):721–736, 2007.

\end{thebibliography}
\end{document}